\documentclass[preprint,12pt]{elsarticle}

\usepackage{amssymb}
\usepackage{amsthm}
\usepackage{ amsfonts}
\usepackage{amsmath}

\newtheorem{thm}{Theorem}
\newtheorem{stat}{Statement}
\newtheorem{lemma}{Lemma}
\newdefinition{defi}{Definition}
\newdefinition{remark}{Remark}
\biboptions{sort,compress}

\journal{Journal of Complexity}

\begin{document}

\begin{frontmatter}

  \title{Upper and lower estimates for numerical integration errors on spheres
    of arbitrary dimension}

\author{Peter Grabner\fnref{thanks}\corref{cor}}
\ead{peter.grabner@tugraz.at}
\author{Tetiana Stepanyuk\fnref{thanks}}
\ead{t.stepaniuk@tugraz.at}
\address{Graz University of Technology, Institute of Analysis and Number
  Theory, Kopernikusgasse 24/II 8010, Graz, Austria}
\fntext[thanks]{The authors are supported by the Austrian Science Fund FWF
  project F5503 (part of the Special Research Program (SFB) 
``Quasi-Monte Carlo Methods: Theory and Applications'')}

\cortext[cor]{Corresponding author}

\begin{abstract}
  In this paper we study the worst-case error of numerical integration on the
  unit sphere $\mathbb{S}^{d}\subset\mathbb{R}^{d+1}$, $d\geq2$, for certain
  spaces of continuous functions on $\mathbb{S}^{d}$. For the classical Sobolev
  spaces $\mathbb{H}^s(\mathbb{S}^d)$ ($s>\frac d2$) upper and lower bounds for
  the worst case integration error have been obtained in
  \cite{Brauchart-Hesse2007:numerical_integration,
    Hesse-Sloan2006:cubature_sobolev,
    Hesse2006:lower_cubature,Hesse-Sloan2005:optimal_lower}. We investigate the
  behaviour for $s\to\frac d2$ by introducing spaces
  $\mathbb{H}^{\frac d2,\gamma}(\mathbb{S}^d)$ with an extra logarithmic
  weight. For these spaces we obtain similar upper and lower bounds for the
  worst case integration error.



%
\end{abstract}


\end{frontmatter}

\section{Introduction}
\label{intro}

Let $\mathbb{S}^{d}\subset\mathbb{R}^{d+1}$, where $d\geq2$ denote the unit
sphere in the Euclidean space $\mathbb{R}^{d+1}$. The integral of a continuous
function $f: \mathbb{S}^{d}\rightarrow\mathbb{R}$, denoted by
\begin{equation*}
\mathrm{I}(f):=\int_{\mathbb{S}^{d}}f(\mathbf{x})d\sigma_{d}(\mathbf{x}),
\end{equation*}
where $d\sigma_{d}(\mathbf{x})$ is the normalised surface (Lebesgue) measure on
$\mathbb{S}^{d}$ (i.e., 
$\sigma_d(\mathbb{S}^d)=1$), is approximated by an
$N$-point numerical integration rule $Q[X_{N},\omega](f)$
\begin{equation*}
  Q[X_{N},\omega](f)=Q[X_{N},(\omega_{j})_{j=1}^{N}](f):=
  \sum\limits_{i=1}^{N}\omega_{i}f(\mathbf{x}_{i})
\end{equation*}
with nodes $\mathbf{x}_{1},\ldots,\mathbf{x}_{N}\in\mathbb{S}^{d}$ and
associated weights $\omega_{1},\ldots,\omega_{N}\in\mathbb{R}$. We will always
assume that the weights satisfy the relation
\begin{equation*}
\sum\limits_{i=1}^{N}\omega_{i}=1.
\end{equation*}

By $Q[X_{N}](f)$ we will denote the equal weight numerical integration rule:
\begin{equation*}
Q[X_{N}](f):=\frac{1}{N}\sum\limits_{i=1}^{N}f(\mathbf{x}_{i}).
\end{equation*}
The worst-case (cubature) error of the cubature rule $Q[X_{N},\omega]$ in a
Banach space $B$ of continuous functions on $\mathbb{S}^{d}$ with norm
$\|\cdot\|_{B}$ is defined by
 \begin{equation}\label{wce}
 \mathrm{wce}(Q[X_{N},\omega];B):=\sup\limits_{f\in B,\|f\|_{B}\leq1}
 |Q[X_{N}](f)-\mathrm{I}(f)|.
 \end{equation}
 In this work we consider reproducing kernel Hilbert spaces
 $\mathbb{H}^{(\frac{d}{2},\gamma)}(\mathbb{S}^{d})$, which interpolate the
 classical spaces $\mathbb{H}^{s}(\mathbb{S}^{d})$ for $s\to\frac d2$, (see
 Section~\ref{prelim} for a precise definition).
 
 The paper is organised as follows.

 Section~\ref{prelim} provides necessary background for Jacobi polynomials, the
 spaces $\mathbb{H}^{s}(\mathbb{S}^{d})$ and
 $\mathbb{H}^{(\frac{d}{2},\gamma)}(\mathbb{S}^{d})$, their associated
 reproducing kernels and an expression for the worst-case error.
 
 In Section~\ref{designs} we find upper and lower bounds of equal weight
 numerical integration over the unit sphere
 $\mathbb{S}^{d}\subset\mathbb{R}^{d+1}$ for functions in the space
 $\mathbb{H}^{(\frac{d}{2},\gamma)}(\mathbb{S}^{d})$, $\gamma>\frac{1}{2}$. In
 this section we consider sequences $X_{N(t)}$ of well-separated
 $t$-designs. Here we also assume that $t\asymp N^{\frac{1}{d}}$. Such
 $t$-designs exist by
 \cite{Bondarenko-Radchenko-Viazovska2015:well_separated}. We write
 $a_{n}\asymp b_{n}$ to mean that there exist positive constants $C_{1}$ and
 $C_{2}$ independent of $n$ such that $C_{1}a_{n}\leq b_{n}\leq C_{2}a_{n}$ for
 all $n$.

We show that
\begin{equation}\label{eq:lower}
  C_{d,\gamma}N^{-\frac{1}{2}}\left(\ln
    N\right)^{-\gamma}\leq
  \mathrm{wce}(Q[X_{N},\omega];\mathbb{H}^{(\frac{d}{2},\gamma)}(\mathbb{S}^{d}))
\end{equation}
for all quadrature rules $Q[X_N,\omega]$ and provide examples of quadrature
rules which satisfy
\begin{equation}\label{eq:wce}
  C_{d,\gamma}^{(1)}N^{-\frac{1}{2}}\left(\ln
    N\right)^{-\gamma+\frac{1}{2}}\leq \mathrm{wce}(Q[X_{N}];\mathbb{H}^{(\frac{d}{2},\gamma)}(\mathbb{S}^{d}))\leq
  C_{d,\gamma}^{(2)}N^{-\frac{1}{2}}\left(\ln
    N\right)^{-\gamma+\frac{1}{2}},
\end{equation}
where the positive constants $C_{d,\gamma}^{(1)}$ and $C_{d,\gamma}^{(2)}$
depend only on $d$ and $\gamma$, but are independent of the rule $Q[X_{N}]$ and
the number of nodes $N$ of the rule.

Here and further by $C_{\gamma,d}$, $C_{\gamma,d}^{(1)}$ and
$C_{\gamma,d}^{(2)}$ we denote some positive constants, which depend only on
$d$ and $\gamma$ and can be different in different relations.

The upper estimate of this result is an extension of results in
\cite{Brauchart-Hesse2007:numerical_integration,
  Hesse-Sloan2006:cubature_sobolev}, where the upper bound for the worst-case
error in the Sobolev space $\mathbb{H}^{s}(\mathbb{S}^{d})$, $s>\frac d2$, (see
Section~\ref{prelim} for a precise definition) of a sequence of cubature rules
$Q[X_{N}]$ was found. In these papers the sequence $Q[X_{N}]$ integrates all
spherical polynomials of degree $\leq t$ exactly and satisfies a certain local
regularity property.

In Section~\ref{lower} we show that the worst-case error for functions in the
space $\mathbb{H}^{(\frac{d}{2},\gamma)}(\mathbb{S}^{d})$,
$\gamma>\frac{1}{2}$, for an arbitrary $N$-point cubature rule
$Q[X_{N}, \omega]$ has the lower bound
\begin{equation*}
  \mathrm{wce}(Q[X_{N},\omega];\mathbb{H}^{(\frac{d}{2},\gamma)}(\mathbb{S}^{d}))
  \geq C_{d,\gamma}N^{-\frac{1}{2}}\left(\ln N\right)^{-\gamma},
\end{equation*}
where the positive constant $C_{d,\gamma}$ depends only on $d$ and $\gamma$,
but is independent of the rule $Q[X_{N}]$ and the number of nodes $N$ of the
rule. On the basis of the estimate \eqref{eq:wce}, we can make a conjecture
that the order of convergence
$\mathcal{O}(N^{-\frac{1}{2}}\left(\ln N\right)^{-\gamma+\frac{1}{2}})$ is
optimal for classes $\mathbb{H}^{(\frac{d}{2},\gamma)}(\mathbb{S}^{d})$.
 
In Section~\ref{QMC} we analyse QMC designs for
$\mathbb{H}^{(\frac{d}{2},\gamma)}(\mathbb{S}^{d})$ and compare them with QMC
designs for Sobolev spaces $\mathbb{H}^{s}(\mathbb{S}^{d})$. We prove that if
$X_{N}$ is a sequence of QMC designs for Sobolev spaces
$\mathbb{H}^{s}(\mathbb{S}^{d})$, $s>\frac{d}{2}$, it is also a sequence of QMC
designs for $\mathbb{H}^{(\frac{d}{2},\gamma)}(\mathbb{S}^{d})$ for all
$\gamma>\frac{1}{2}$.
 
We remark here that J.~Beck \cite{Beck_Chen1987:irregularities_distribution,
  Beck1984:sums_distances_between} could show a lower bound for the spherical
cap discrepancy of order $N^{-1/2-1/2d}$; he proved by probabilistic means that
for every $N$ there exists a point set $X_N$ with discrepancy of order
$N^{-1/2-1/2d}\sqrt{\log N}$. Beck's lower bound can be reproved by using the
techniques found by D.~Bilyk and F.~Dai \cite{Bilyk-Dai2016:geodesic_riesz},
which we will refer to in more detail in Section~\ref{lower}. The $\sqrt{\log
  N}$-factor between the lower and the upper bound in \eqref{eq:lower} and
\eqref{eq:wce} resembles the difference between Beck's general lower bound and
the upper bound achieved by a probabilistic construction.
\section{Preliminaries}
\label{prelim}

\subsection{Background and basic notations}

We denote the Euclidean inner product of $\mathbf{x}$ and $\mathbf{y}$ in
$\mathbb{R}^{d+1}$ by $\langle\mathbf{x},\mathbf{y}\rangle$.

We use the Pochhammer symbol $(a)_{n}$, where
$n\in \mathbb{N}_{0}$ and $a\in \mathbb{R}$, defined by
\begin{equation*}
  (a)_{0}:=1, \quad (a)_{n}:=a(a+1)\ldots(a+n-1)\quad \mathrm{for} \quad
  n\in \mathbb{N},
\end{equation*}
which can be  written in the terms of the gamma function $\Gamma(z)$ by means of
\begin{equation}\label{Pochhammer}
 (a)_{\ell}=\frac{\Gamma(\ell+a)}{\Gamma(a)}.
 \end{equation}
 The following asymptotic  relation holds
 \begin{equation}\label{gamma}
 \frac{\Gamma(z+a)}{\Gamma(z+b)}\sim z^{a-b}\quad \text{as } z\rightarrow\infty
 \quad\text{in the sector } |\arg z|\leq\pi-\delta
\end{equation}
for $\delta>0$.  Here, $f(x)\sim g(x)$, $x\rightarrow\infty$, means that
 \begin{equation*}
 \lim\limits_{x\rightarrow\infty}\frac{f(x)}{g(x)}=1.
 \end{equation*}
 
We denote, as usual, by $\{Y_{\ell,k}^{(d)}: k=1,\ldots, Z(d,\ell)\}$ a collection of 
$\mathbb{L}_{2}$-orthonormal real spherical harmonics (homogeneous harmonic  polynomials in $d+1$ variables  restricted to $\mathbb{S}^{d}$) of degree $\ell$ (see, e.g., \cite{Mueller1966:spherical_harmonics}).
The space of spherical harmonics of degree $\ell\in\mathbb{N}_{0}$ on 
$\mathbb{S}^{d}$ has the dimension
\begin{equation}\label{Zd}
  Z(d,0)=1,\quad Z(d,\ell)=(2\ell+d-1)
  \dfrac{\Gamma(\ell+d-1)}{\Gamma(d)\Gamma(\ell+1)}\sim
  \frac{2}{\Gamma(d)}\ell^{d-1}, \quad \ell\rightarrow\infty.
 \end{equation}
Each spherical harmonic  $Y_{\ell,k}^{(d)}$ of exact degree $\ell$ is an eigenfunction of the negative Laplace-Beltrami operator $-\Delta^{*}_{d}$ with eigenvalue
\begin{equation}\label{eigenvalue}
 \lambda_{\ell}:=\ell(\ell+d-1).
 \end{equation}
 
The spherical harmonics of degree $\ell$ satisfy the addition theorem: 
\begin{equation}\label{additiontheorem}
 \sum\limits_{k=1}^{Z(d,\ell)}Y_{\ell,k}^{(d)}(\mathbf{x})Y_{\ell,k}^{(d)}(\mathbf{y})=Z(d,\ell)P_{\ell}^{(d)}(\langle\mathbf{x},\mathbf{y}\rangle),
 \end{equation}
 where $P_{\ell}^{(d)}$ is the $\ell$-th generalised Legendre polynomial,
 normalised by ${P_{\ell}^{(d)}(1)=1}$ and orthogonal on the interval $[-1,1]$
 with respect to the weight function $(1-t^{2})^{d/2-1}$. These functions are
 zonal spherical harmonics on $\mathbb{S}^{d}$. Notice that
\begin{equation}\label{Jacobigegenbauer}
 Z(d,n)P_{n}^{(d)}(x)=\frac{n+\lambda}{\lambda}C_{n}^{\lambda}(x), \ \ \ \ 
  P_{n}^{(d)}(x)=\frac{n!}{(d/2)_{n}}P_{n}^{(\frac{d}{2}-1, \frac{d}{2}-1)}(x),
  \end{equation}
 where $C_{n}^{\lambda}(x)$ is the $n$-th Gegenbauer polynomial with index
 $\lambda=\dfrac{d-1}{2}$ and $P_{n}^{(\frac{d}{2}-1,\frac{d}{2}-1)}(x)$ are
 the Jacobi polynomials with the indices $\alpha=\beta= \frac{d}{2}-1$.

\subsection{Jacobi polynomials}
The Jacobi polynomials $P_{\ell}^{(\alpha,\beta)}(x)$ are the polynomials
orthogonal over the interval $[-1,1]$ with respect to the weight function
$w_{\alpha,\beta}(x)=(1-x)^{\alpha}(1+x)^{\beta}$ and normalised by the
relation
\begin{equation}\label{JacobiMax}
  P_{\ell}^{(\alpha,\beta)}(1)=\binom {\ell+\alpha}\ell=
  \frac{(1+\alpha)_{\ell}}{\ell!}\sim\frac{1}{\Gamma(1+\alpha)}\ell^{\alpha},
  \quad \alpha,\beta>-1.
 \end{equation}
 (see, e.g., \cite[(5.2.1)]{Magnus-Oberhettinger-Soni1966:formulas_theorems}).

 Also the following equality holds
\begin{equation}\label{JacobiMinus}
 P_{\ell}^{(\alpha,\beta)}(-x)=(-1)^{\ell} P_{\ell}^{(\beta,\alpha)}(x).
 \end{equation}

 For fixed ${\alpha, \beta>-1}$ and ${0< \theta<\pi}$, the following relation
 gives an asymptotic approximation for $\ell\rightarrow\infty$ (see,
 e.g.,\cite[Theorem 8.21.13]{Szegoe1975:orthogonal_polynomials})
\begin{multline*}
  P_{\ell}^{(\alpha,\beta)}(\cos \theta)=\frac{1}{\sqrt{\pi}}\ell^{-1/2}
  \Big(\sin\frac{\theta}{2}\Big)^{-\alpha-1/2}
  \Big(\cos\frac{\theta}{2}\Big)^{-\beta-1/2}\\
  \times\Big\{\cos \Big(\Big(\ell+\frac{\alpha+\beta+1}{2}\Big)\theta-
  \frac{2\alpha+1}{4}\pi\Big)+\mathcal{O}(\ell\sin\theta)^{-1}\Big\}.
\end{multline*}
Thus, for
$c_{\alpha,\beta}\ell^{-1}\leq\theta\leq\pi-c_{\alpha,\beta}\ell^{-1}$ the last
asymptotic equality yields
 \begin{equation}\label{JacobiIneq}
   |P_{\ell}^{(\alpha,\beta)}(\cos \theta)|\leq \tilde{c}_{\alpha,\beta}
   \ell^{-1/2}(\sin\theta)^{-\alpha-1/2}+
 \tilde{c}_{\alpha,\beta}\ell^{-3/2}(\sin\theta)^{-\alpha-3/2}, \quad\alpha\geq\beta.
 \end{equation}
 
 If $\alpha,\beta$ are real and $c$ is fixed positive constant, then as
 $\ell\rightarrow\infty$ (see, e.g.,
 \cite[(5.2.3)]{Magnus-Oberhettinger-Soni1966:formulas_theorems})
\begin{equation}\label{JacobiIneq1}
|P_{\ell}^{(\alpha,\beta)}(\cos \theta)|=\begin{cases}
  \mathcal{O}\big(\theta^{-\frac{1}{2}-\alpha}\ell^{-\frac{1}{2}}\big) & \text{if }
  \frac{c}{\ell}\leq \theta\leq \frac{\pi}{2}, \\
\mathcal{O}(\ell^{\alpha}) & \text{if }  0 \leq \theta\leq\frac{c}{\ell}.
  \end{cases}
\end{equation}
 
We will also use the formula (see, e.g.,
\cite[(4.5.3)]{Szegoe1975:orthogonal_polynomials})
  \begin{equation}\label{ChristoffelDarbu}
 \sum\limits_{\ell=0}^{n}\frac{2\ell+\alpha+\beta+1}{\alpha+\beta+1}\frac{(\alpha+\beta+1)_{\ell}}{(\beta+1)_{\ell}}P_{\ell}^{(\alpha, \beta)}(t)=
  \frac{(\alpha+\beta+2)_{n}}{(\beta+1)_{n}}P_{n}^{(\alpha+1,\beta)}(t).
\end{equation}

Choosing $\alpha=\beta=\frac{d}{2}-1$ and taking into account the relations
(\ref{Zd}) and (\ref{Jacobigegenbauer}), formula (\ref{ChristoffelDarbu}) also
reads
\begin{equation}\label{ChristoffelDarbu1}
 \sum\limits_{r=0}^{\ell}Z(d,r)P_{r}^{(d)}(t)=
  \sum\limits_{r=0}^{\ell}\frac{2r+d-1}{d-1}\frac{(d-1)_{r}}{(d/2)_{r}}P_{r}^{(\frac{d}{2}-1,\frac{d}{2}-1)}(t)
=
  \frac{(d)_{\ell}}{(d/2)_{\ell}}P_{\ell}^{(\frac{d}{2}, \frac{d}{2}-1)}(t).
\end{equation}
Substituting  $\alpha=\frac{d}{2}-1+k$ and $\beta=\frac{d}{2}-1$, formula (\ref{ChristoffelDarbu}) gives
\begin{equation}\label{ChristoffelDarbu2}
  \sum\limits_{r=0}^{\ell}\frac{2r+d-1+k}{d-1+k}\frac{(d-1+k)_{r}}{(d/2)_{r}}
  P_{r}^{(\frac{d}{2}-1+k,\frac{d}{2}-1)}(t)=
  \frac{(d+k)_{\ell}}{(d/2)_{\ell}}P_{\ell}^{(\frac{d}{2}+k,\frac{d}{2}-1)}(t).
\end{equation}

For any integrable function $f: [-1, 1]\rightarrow \mathbb{R}$ (see, e.g.,
\cite{Mueller1966:spherical_harmonics})
\begin{equation}\label{a1}
 \int\limits_{\mathbb{S}^{d}}f(\langle\mathbf{x},\mathbf{y}\rangle)d\sigma_{d}(\mathbf{x})=\frac{\Gamma(\frac{d+1}{2})}{\sqrt{\pi}\Gamma(\frac{d}{2})}\int\limits_{-1}^{1}f(t)(1-t^{2})^{\frac{d}{2}-1}dt \quad \forall \mathbf{y}\in \mathbb{S}^{d}.
\end{equation}
 For $\alpha>1$ and $L\in \mathbb{N}_{0}$, we have (see, e.g., formula (2.18) in  \cite{Brauchart-Hesse2007:numerical_integration})
 \begin{equation}\label{a2}
 \int\limits_{-1}^{1}P_{\ell}^{(\alpha+L,\alpha)}(t)(1-t^{2})^{\alpha}dt=2^{2\alpha+1}\frac{(L)_{\ell}}{\ell!}\frac{\Gamma(\alpha+1)\Gamma(\alpha+\ell+1)}{\Gamma(2\alpha+\ell+2)}.
\end{equation}
 This formula also can be easily derived with the help of Rodrigues' formula (see, e.g., \cite[(4.3.1)]{Szegoe1975:orthogonal_polynomials}).
 
In particular (\ref{a1}), (\ref{a2}) and (\ref{gamma}) imply
\begin{multline}\label{a3}
\int\limits_{\mathbb{S}^{d}}P_{\ell}^{(\frac{d}{2}+L,\frac{d}{2}-1)}(\langle\mathbf{x},\mathbf{y}\rangle)d\sigma_{d}(\mathbf{x})\\
=2^{d-1}\frac{\Gamma(\frac{d+1}{2})}{\sqrt{\pi}}\frac{(L+1)_{\ell}}{\ell!}\frac{\Gamma(\frac{d}{2}+\ell)}{\Gamma(d+\ell)}
\asymp 
\frac{\Gamma(L+\ell+1)}{\Gamma(\ell+1)}\frac{\Gamma(\frac{d}{2}+\ell)}{\Gamma(d+\ell)}
\asymp \ell^{L-\frac{d}{2}}.
\end{multline}

\subsection{The space of continuous functions on
  $\mathbb{S}^{d}$ and representation of
  worst-case error }
 
The Sobolev space  $\mathbb{H}^{s}(\mathbb{S}^{d})$ 
 for $s\geq 0$ consists of all functions $f\in\mathbb{L}_{2}(\mathbb{S}^{d})$ 
with finite norm 
 \begin{equation}\label{normSobolev}
\|f\|_{\mathbb{H}^{s}}=\bigg(\sum\limits_{\ell=0}^{\infty}\sum\limits_{k=1}^{Z(d,\ell)}
\left(1+\lambda_{\ell}\right)^{s}|\hat{f}_{\ell,k}|^{2}\bigg)^{\frac{1}{2}}, 
\end{equation}
 where the Laplace-Fourier coefficients are given by the formula
\begin{equation*}
  \hat{f}_{\ell,k}:=(f,Y_{\ell,k}^{(d)})_{\mathbb{S}^{d}}=
  \int_{\mathbb{S}^{d}}f(\mathbf{x})Y_{\ell,k}^{(d)}(\mathbf{x})
  d\sigma_{d}(\mathbf{x}).
\end{equation*}

For $s>\frac{d}{2}$ the space $\mathbb{H}^{s}(\mathbb{S}^{d})$ is embedded into
the space of continuous functions $\mathbb{C}(\mathbb{S}^{d})$.  This fact also
implies that point evaluation in $\mathbb{H}^{s}(\mathbb{S}^{d})$,
$s>\frac{d}{2}$, is bounded and $\mathbb{H}^{s}(\mathbb{S}^{d})$,
$s>\frac{d}{2}$, is a reproducing kernel Hilbert space.

In the row of papers \cite{Brauchart-Saff-Sloan+2014:qmc_designs,
  Brauchart-Hesse2007:numerical_integration, Hesse-Sloan2006:cubature_sobolev,
  Hesse2006:lower_cubature,Hesse-Sloan2005:optimal_lower}, the worst-case error
for Sobolev spaces $\mathbb{H}^{s}(\mathbb{S}^{d})$ in the case $s>\frac{d}{2}$
was studied. Our aim is to consider the class of functions, which are less
smooth than functions from $\mathbb{H}^{s}(\mathbb{S}^{d})$, $s>\frac{d}{2}$.

We define the space $\mathbb{H}^{(\frac{d}{2},\gamma)}(\mathbb{S}^{d})$ for
$\gamma>\dfrac{1}{2}$ as the set of all functions
$f\in\mathbb{L}_{2}(\mathbb{S}^{d})$ whose Laplace-Fourier coefficients satisfy
\begin{equation}\label{norm}
\|f\|_{\mathbb{H}^{(\frac{d}{2},\gamma)}}^2:=\sum\limits_{\ell=0}^{\infty}
w_{\ell}(d,\gamma)\sum\limits_{k=1}^{Z(d,\ell)}|\hat{f}_{\ell,k}|^{2}<\infty, 
\end{equation}
where
\begin{equation*}
 w_{\ell}(d,\gamma):=\left(1+\lambda_{\ell}\right)^{\frac{d}{2}}\left(\ln\left(3+\lambda_{\ell}\right)\right)^{2\gamma}.
\end{equation*}

The space $\mathbb{H}^{(\frac{d}{2},\gamma)}(\mathbb{S}^{d})$ is a Hilbert
space with a corresponding inner product denoted by
$(f,g)_{\mathbb{H}^{(\frac{d}{2},\gamma)}}$.  For $\gamma>\frac{1}{2}$ the
space $\mathbb{H}^{(\frac{d}{2},\gamma)}(\mathbb{S}^{d})$ is embedded into the
space of continuous functions $\mathbb{C}(\mathbb{S}^{d})$. Indeed, using the
Cauchy-Schwarz inequality we can show in the same way as in
\cite{Hesse2006:lower_cubature}, that for
$f\in\mathbb{H}^{(\frac{d}{2},\gamma)}(\mathbb{S}^{d})$
\begin{equation*}
\sup\limits_{\mathbf{x}\in\mathbb{S}^{d}}|f(\mathbf{x})|
\leq
 C_{d,\gamma}\|f\|_{\mathbb{H}^{(\frac{d}{2},\gamma)}}.
\end{equation*}

Embedding into $\mathbb{C}(\mathbb{S}^{d})$ implies that point evaluation in
$\mathbb{H}^{(\frac{d}{2},\gamma)}(\mathbb{S}^{d})$ with $\gamma>\frac{1}{2}$
is bounded, and consequently
$\mathbb{H}^{(\frac{d}{2},\gamma)}(\mathbb{S}^{d})$ is a reproducing kernel
Hilbert space. That is to say there exists a kernel $K_{d,\gamma}:$
$\mathbb{S}^{d}\times\mathbb{S}^{d}\rightarrow\mathbb{R}$, with the following
properties: (i)
$K_{d,\gamma}(\mathbf{x},\mathbf{y})=K_{d,\gamma}(\mathbf{y},\mathbf{x})$ for
all $\mathbf{x},\mathbf{y}\in \mathbb{S}^{d}$; (ii)
${K_{d,\gamma}(\cdot,\mathbf{x})\in
  \mathbb{H}^{(\frac{d}{2},\gamma)}(\mathbb{S}^{d})}$ for all fixed
$\mathbf{x}\in\mathbb{H}^{(\frac{d}{2},\gamma)}(\mathbb{S}^{d})$; and (iii) the
reproducing property
\begin{equation*}
(f,K_{d,\gamma}(\cdot,\mathbf{x}))_{\mathbb{H}^{(\frac{d}{2},\gamma)}}=f(\mathbf{x})
\quad \forall f\in\mathbb{H}^{(\frac{d}{2},\gamma)}(\mathbb{S}^{d}) \quad \forall
\mathbf{x} \in \mathbb{S}^{d}.
\end{equation*}

The reproducing kernel $K_{d,\gamma}$ in
$\mathbb{H}^{(\frac{d}{2},\gamma)}(\mathbb{S}^{d})$ is given by
\begin{equation}\label{kernel}
K_{d,\gamma}(\mathbf{x},\mathbf{y})
=\sum\limits_{\ell=0}^{\infty}w_{\ell}(d,\gamma)^{-1}Z(d,\ell)
 P_{\ell}^{(d)}(\langle\mathbf{x},\mathbf{y}\rangle).
 \end{equation}
 It is easily verified, that $K_{\gamma,d}$, defined by (\ref{kernel}) has the
 reproducing kernel properties.

 This kernel is a zonal function: $K_{\gamma,d}(\mathbf{x},\mathbf{y})$ depends
 only on the inner product $\langle\mathbf{x},\mathbf{y}\rangle$.

 Using arguments, as in (\cite{Brauchart-Saff-Sloan+2014:qmc_designs} or
 \cite{Hesse-Sloan2006:cubature_sobolev}), it is possible to write down an
 expression for the worst-case error. Indeed
 \begin{equation*}
 \mathrm{wce}(Q[X_{N},\omega];\mathbb{H}^{(\frac{d}{2},\gamma)}(\mathbb{S}^{d}))^{2}
=\sum\limits_{i,j=1}^{N}\omega_{i}\omega_{j}K_{d,\gamma}(\mathbf{x}_{i},\mathbf{x}_{j})-\int_{\mathbb{S}^{d}}K_{d,\gamma}(\mathbf{x},\mathbf{y})d\sigma_{d}(\mathbf{y}),
 \end{equation*}
 where we have used the reproducing property of $K_{d,\gamma}$.
 
 Therefore,
  \begin{equation}\label{wceKernel}
\mathrm{wce}(Q[X_{N},\omega];\mathbb{H}^{(\frac{d}{2},\gamma)}(\mathbb{S}^{d}))^{2}=\sum\limits_{i,j=1}^{N}\omega_{i}\omega_{j}\tilde{K}_{d,\gamma}(\mathbf{x}_{i},\mathbf{x}_{j}),
 \end{equation}
  where
 \begin{equation}\label{kernel0}
 \tilde{K}_{d,\gamma}(\mathbf{x},\mathbf{y})=\sum\limits_{\ell=1}^{\infty}w_{\ell}(d,\gamma)^{-1}Z(d,\ell)
 P_{\ell}^{(d)}(\langle\mathbf{x},\mathbf{y}\rangle).
 \end{equation}
 
%

\section{Upper and lower bounds for the  worst-case error for well-separated $t$-designs}
\label{designs}
\begin{defi}\label{def1}
  A spherical $t$-design is a finite subset $X_{N}\subset \mathbb{S}^{d}$ with
  the characterising property that an equal weight integration rule with nodes
  from $X_{N}$ integrates all polynomials $p$ with degree $\leq t$ exactly;
  that is,
\begin{equation*}
  \frac{1}{N}\sum\limits_{\mathbf{x}\in X_{N}}p(\mathbf{x})=
  \int_{\mathbb{S}^{d}}p(\mathbf{x})d\sigma_{d}(\mathbf{x}), \quad
  \mathrm{deg}(p)\leq t.
\end{equation*}
Here $N$ is the number of points of the spherical design.
\end{defi}

A concept of $t$-design was introduced in the paper
\cite{Delsarte-Goethals-Seidel1977:spherical_designs} by Delsarte, Goethals and
Seidel. There it was proved that the number of points for a $t$-design has to
satisfy $N\geq C_dt^d$ for a positive constant $C_d$.

Bondarenko, Radchenko and Viazovska
\cite{Bondarenko-Radchenko-Viazovska2013:optimal_designs} proved that there
always exist spherical $t$-designs with $N\asymp t^{d}$ points.  That is why in
this section we always assume that
 \begin{equation}\label{constantc0}
 N=N(t)\asymp t^{d}.
 \end{equation} 
 Then
 \begin{equation*}
\frac{1}{N^{2}}\sum\limits_{i=1}^{N}\sum\limits_{j=1}^{N}
P_{\ell}^{(d)}(\langle\mathbf{x}_{i},\mathbf{x}_{j}\rangle)=0,
\quad\text{for }\ell=1,\ldots,t.
 \end{equation*}
Thus for such sequences $Q[X_{N(t)}]$ (\ref{wceKernel})  simplifies to
\begin{equation}\label{wceTdesign}
\mathrm{wce}(Q[X_{N(t)}];\mathbb{H}^{(\frac{d}{2},\gamma)}(\mathbb{S}^{d}))^{2}
=\frac{1}{N^{2}}\sum\limits_{i,j=1}^{N}\sum\limits_{\ell=t+1}^{\infty}w_{\ell}(d,\gamma)^{-1}Z(d,\ell)
 P_{\ell}^{(d)}(\langle\mathbf{x}_{i},\mathbf{x}_{j}\rangle).
\end{equation}

By a spherical cap $S(\mathbf{x}; \varphi)$ of centre $\mathbf{x}$ and angular
radius $\varphi$ we mean
 \begin{equation*}
 S(\mathbf{x}; \varphi):=\big\{\mathbf{y}\in \mathbb{S}^{d} \big| \langle\mathbf{x},\mathbf{y}\rangle\geq \cos\varphi \big\}.
 \end{equation*}
The normalised surface area of a spherical cap is given by 
\begin{equation}\label{capArea}
|S(\mathbf{x}; \varphi)|=\frac{\Gamma((d+1)/2)}{\sqrt{\pi}\Gamma(d/2)}
\int\limits_{\cos\varphi}^{1}(1-t^{2})^{\frac{d}{2}-1}dt
\asymp(1-\cos\varphi)^{\frac{d}{2}} \quad\text{as } \varphi\rightarrow 0.
\end{equation}
Here and in the sequel we use $|S|$ as a shorthand for $\sigma_d(S)$ for
$S\subset\mathbb{S}^d$.

\begin{defi}[Property (R)]
  A sequence $(Q[X_{N(t)},\omega])_{t\in\mathbb{N}}$ of numerical integration
  rules $Q[X_{N(t)},\omega]$, which integrates all spherical polynomials of
  degree $\leq t$ exactly, that is
\begin{equation*}
  \sum\limits_{j=1}^{N(t)}\omega_{j}p(\mathbf{x}_{j})=
  \int\limits_{\mathbb{S}^{d}}p(\mathbf{x})d\sigma_{d}(\mathbf{x}), \quad
  \mathrm{deg}(p)\leq t.
\end{equation*}
is said to have property (R) (or to be ``quadrature regular''), if there
exist positive numbers $c_{1}$ and $c_{2}$ independent of $t$ with
$c_{1}\leq\pi/2$, such that for all $t\geq1$ the weights $\omega_{j}$
associated with the nodes $\mathbf{x}_{j}, j=1,\ldots,N(t)$ of $Q[X_{N(t)}]$
satisfy
  \begin{equation}\label{propertyR}
{\mathop{\sum}\limits_{
 j=1,\atop \mathbf{x}_{j} \in S(\mathbf{x};\frac{c_{1}}{t})}^{N(t)}} 
|\omega_{j}| \leq c_{2}\Big| S(\mathbf{x};\frac{c_{1}}{t}) \Big|\quad
\forall \mathbf{x}\in \mathbb{S}^{d}.
 \end{equation}
 \end{defi}

 Reimer \cite{Reimer2000:hyperinterpolation_sphere} has shown that every
 sequence of positive weight cubature rules $Q[X_{N(t)},\omega]$, with
 $Q[X_{N(t)},\omega](p)=\mathrm{I}(p)$ for all polynomials $p$ with
 $\deg p\leq t$ satisfies property (R) automatically with positive constants
 $c_{1}$ and $c_{2}$ depending only on $d$.

\begin{defi}
  A sequence of $N$-point sets $X_{N}$,
  $X_{N}=\big\{\mathbf{x}_{1},\ldots, \mathbf{x}_{N} \big\}$, is called
  well-separated if there exists a positive constant $c_{3}$ such that
\begin{equation}\label{condiii}
  \min\limits_{i\neq j}|\mathbf{x}_{i}- \mathbf{x}_{j}|>
  \frac{c_{3}}{N^{\frac{1}{d}}}.
\end{equation}
\end{defi}
 
It should be noticed, that a well-separated sequence $X_{N}$ of numerical
integration rules with equal weights $\omega_{i}=\frac{1}{N}$ satisfies
property (R), but not conversely.  Indeed, from the inequality (\ref{condiii})
it follows, that for all ${\mathbf{x}_{i}, \mathbf{x}_{j} \in X_{N(t)}}$,
$i\neq j$,
\begin{equation*}
\langle\mathbf{x}_{i}, \mathbf{x}_{j}\rangle<1-\frac{c^{2}_{3}}{2N^{\frac{2}{d}}}.
\end{equation*}
Thus the spherical cap
$S\Big(\mathbf{x}_{i};\arccos\Big(1-\frac{c^{2}_{3}}{2N^{\frac{2}{d}}}\Big)\Big)$
contains no points of $X_{N}$ in its interior, except of the point
$\mathbf{x}_{i}$.

Using (\ref{capArea}) we deduce the following estimate
\begin{equation*}
\frac{1}{N}  \# \Big\{ \mathbf{x}_{j}\in X_{N(t)}
\cap S(\mathbf{x};\frac{c_{1}}{t})\Big\}
 \leq 
 \frac{1}{N}\frac{\Big| S(\mathbf{x};\frac{c_{1}}{t}) \Big|}
 {\Big|S\Big(\mathbf{x};\arccos\Big(1-\frac{c^{2}_{3}}
   {2N^{\frac{2}{d}}}\Big)\Big) \Big|}\ll
  \Big| S(\mathbf{x};\frac{c_{1}}{t}) \Big|.
\end{equation*}

 Here we write $a_{n}\ll b_{n}$ ($a_{n}\gg b_{n}$) to mean that there exists positive constant $K$  independent of $n$ such that 
 $a_{n}\leq  K b_{n}$ ($a_{n}\geq K b_{n}$) for all $n$.

\begin{thm}\label{thm1} Let $d\geq2$, $\gamma>\dfrac{1}{2}$ be fixed, and
  $(X_{N(t)})_t$ be a sequence  be a well-separated spherical $t$-designs, 
  $t$ and $N(t)$ satisfy relation (\ref{constantc0}).
Then there exist  positive constants 
$C_{d,\gamma}^{(1)} $ and $C_{d,\gamma}^{(2)}$, such that
 \begin{equation}\label{theorem1}
   C_{d,\gamma}^{(1)}N^{-\frac{1}{2}}\left(\ln N\right)^{-\gamma+\frac{1}{2}}\leq
   \mathrm{wce}(Q[X_{N}];\mathbb{H}^{(\frac{d}{2},\gamma)}(\mathbb{S}^{d}))
   \leq C_{d,\gamma}^{(2)}N^{-\frac{1}{2}}\left(\ln N\right)^{-\gamma+\frac{1}{2}}.
\end{equation}
The constants $C_{d,\gamma}^{(1)} $ and $C_{d,\gamma}^{(2)} $ depend only on
$d$, $\gamma$ and on the constants $c_{i}$, $i=1,\ldots,3$, from the relations
(\ref{propertyR}) and (\ref{condiii}).
\end{thm}



In (\ref{theorem1}) and further in this section for brevity we write $N$
instead $N(t)$ for the number of nodes in $X_{N(t)}$.

Theorem~\ref{thm1} is a consequence of the following lemmas:

\begin{lemma}\label{lem1}
Let $d\geq2$ and $\gamma>\dfrac{1}{2}$  be fixed.
Then for any sequence $X_{N}$, $K\in \mathbb{N}_{0}$ and for any
$n\in\mathbb{N}$ the following relation holds
\begin{align}\label{Lemma1}
  &\frac{1}{N^{2}}\sum\limits_{i,j=1}^{N}\sum\limits_{\ell=n+1}^{\infty}
  w_{\ell}(d,\gamma)^{-1}Z(d,\ell)
  P_{\ell}^{(d)}(\langle\mathbf{x}_{i},\mathbf{x}_{j}\rangle)\\
  \ll
  &\frac{1}{N^{2}}\sum\limits_{i,j=1}^{N}\sum\limits_{\ell=n+1}^{\infty}
    \ell^{-\frac{d}{2}-K}(\ln \ell)^{-2\gamma}
P_{\ell}^{(\frac{d}{2}+K-1,\frac{d}{2}-1)}(\langle\mathbf{x}_{i},\mathbf{x}_{j}\rangle).\notag
\end{align}
\end{lemma}

\begin{lemma}\label{lem2} Let $d\geq2$ and $\gamma>\dfrac{1}{2}$  be fixed, let
  $(X_{N(t)})_t$ be a sequence of spherical $t$-designs, $t$ and $N(t)$ satisfy
  relation (\ref{constantc0}).  Then for any $K\in \mathbb{N}_{0}$ there exists
  a positive constant $C_{d,\gamma} $, such that
\begin{align}
&\frac{1}{N^{2}}\sum\limits_{i,j=1}^{N}\sum\limits_{\ell=t+1}^{\infty} \ell^{-\frac{d}{2}-K}(\ln \ell)^{-2\gamma}
 P_{\ell}^{(\frac{d}{2}+K-1, \frac{d}{2}-1)}(\langle\mathbf{x}_{i},\mathbf{x}_{j}\rangle)-
C_{d,\gamma} t^{-d}(\ln t)^{-2\gamma}\notag\\
&\ll\mathrm{wce}(Q[X_{N}];\mathbb{H}^{(\frac{d}{2},\gamma)})^{2}\label{Lemma2}\\
&\ll\frac{1}{N^{2}}\sum\limits_{i,j=1}^{N}\sum\limits_{\ell=t+1}^{\infty} \ell^{-\frac{d}{2}-K}(\ln \ell)^{-2\gamma}
 P_{\ell}^{(\frac{d}{2}+K-1, \frac{d}{2}-1)}(\langle\mathbf{x}_{i},\mathbf{x}_{j}\rangle).\notag
\end{align}
The constant $C_{d,\gamma} $  depends only on  $d$ and $\gamma$.
\end{lemma}

\begin{lemma}\label{lem3} Let $d\geq2$ and  $\gamma>\dfrac{1}{2}$ 
  be fixed, $(X_{N(t)})_{t}$ be a well-separated sequence, $t$ and $N(t)$
  satisfy relation (\ref{constantc0}).  Then for any $K>\frac{d}{2}$,
  $K\in \mathbb{N}$, there exist positive constants $C_{d,\gamma}^{(1)} $ and
  $C_{d,\gamma}^{(2)}$, such that
\begin{align}\label{Lemma3}
&C_{d,\gamma}^{(1)}N^{-1}\left(\ln N\right)^{-2\gamma+1}\\
&\leq\frac{1}{N^{2}}\sum\limits_{i,j=1}^{N}\sum\limits_{\ell=t+1}^{\infty} \ell^{-\frac{d}{2}-K}(\ln \ell)^{-2\gamma}
 P_{\ell}^{(\frac{d}{2}+K-1, \frac{d}{2}-1)}(\langle\mathbf{x}_{i},\mathbf{x}_{j}\rangle)
\leq C_{d,\gamma}^{(2)} N^{-1}\left(\ln N\right)^{-2\gamma+1}.
\notag
\end{align}
The constants 
$C_{d,\gamma}^{(1)} $ and $C_{d,\gamma}^{(2)}$ depend only on  $d$ and
$\gamma$.
\end{lemma}

\begin{remark}\label{rem1} Let $d\geq2$, $\gamma>\dfrac{1}{2}$ be fixed and let
  the sequence $(X_{N})_N$ have property (R).  Then
\begin{equation}\label{Remark1}
\frac{1}{N^{2}}\sum\limits_{i,j=1}^{N}\sum\limits_{\ell=[N^{\frac{1}{d}}]+1}^{\infty}\ell^{-\frac{d}{2}-K}(\ln \ell)^{-2\gamma}
 P_{\ell}^{(\frac{d}{2}+K-1, \frac{d}{2}-1)}(\langle\mathbf{x}_{i},\mathbf{x}_{j}\rangle)
 \ll N^{-1}\left(\ln N\right)^{-2\gamma+1}.
\end{equation}
\end{remark}

Lemma~\ref{lem1} and Remark~\ref{rem1} allow us to write down the following
estimate.

\begin{thm}\label{thm2} Let $d\geq2$, $\gamma>\dfrac{1}{2}$ be fixed and let
  the sequence $(X_{N})_N$ have property (R).
  Then
\begin{align}\label{theorem2}
  &\mathrm{wce}(Q[X_{N}];\mathbb{H}^{(\frac{d}{2},\gamma)}(\mathbb{S}^{d}))^{2}
    \notag\\
=&\frac{1}{N^{2}}\sum\limits_{i,j=1}^{N}\sum\limits_{\ell=1}^{[N^{\frac{1}{d}}]} w_{\ell}(d,\gamma)^{-1}Z(d,\ell)
 P_{\ell}^{(d)}(\langle\mathbf{x}_{i},\mathbf{x}_{j}\rangle)+\mathcal{O}\Big(\frac{1}{N (\ln N)^{2\gamma-1}}\Big).\notag
\end{align}
\end{thm}

From the proofs of Lemmas 1--3 one can easily get an estimate.

\begin{thm}\label{thm3} Let $d\geq2$, $\gamma>\dfrac{1}{2}$ be fixed and let
  $(X_{N(t)})_t$ be a sequence of spherical $t$-designs.
Then there exists  a positive constant
 $C_{d,\gamma}$ such that
 \begin{equation}\label{theorem3}
\mathrm{wce}(Q[X_{N}];\mathbb{H}^{(\frac{d}{2},\gamma)}(\mathbb{S}^{d}))\leq C_{d,\gamma}t^{-\frac{d}{2}}\left(\ln t\right)^{-\gamma+\frac{1}{2}}.
\end{equation}
The constant  $C_{d,\gamma} $ depends only on  $d$ and $\gamma$.
\end{thm}

\begin{proof}[Proof of  Lemma  1]
We write  \begin{equation*}
 \Delta a_{\ell}:=a_{\ell}-a_{\ell+1}.
 \end{equation*}

For all $K\in \mathbb{N}_{0}$ denote by $a_{\ell}^{(K)}$ the following quantity
\begin{equation}\label{alK}
a_{\ell}^{(K)}=a_{\ell}^{(K)}(\gamma,d):=\left(1+\lambda_{\ell}\right)^{-\frac{d}{2}-K}\left(\ln\left(3+\lambda_{\ell}\right)\right)^{-2\gamma}.
\end{equation}

An application of Abel summation yields
\begin{align}
&\frac{1}{N^{2}}\sum\limits_{i,j=1}^{N}\sum\limits_{\ell=n+1}^{\infty}a_{\ell}^{(0)}Z(d,\ell)
  P_{\ell}^{(d)}(\langle\mathbf{x}_{i},\mathbf{x}_{j}\rangle)\notag\\
  =&\frac{1}{N^{2}}\sum\limits_{i,j=1}^{N}\sum\limits_{\ell=n+1}^{\infty}\Delta a_{\ell}^{(0)}\sum\limits_{k=0}^{\ell}Z(d,k)
 P_{k}^{(d)}(\langle\mathbf{x}_{i},\mathbf{x}_{j}\rangle)\label{for112}\\
 -&a_{n+1}^{(0)}\sum\limits_{k=0}^{n}Z(d,k)
 \frac{1}{N^{2}}\sum\limits_{i,j=1}^{N}P_{k}^{(d)}(\langle\mathbf{x}_{i},\mathbf{x}_{j}\rangle).\notag
 \end{align}
  
Here and further we use  that for all $k,\ell\in\mathbb{N}_{0}$
 \begin{equation}\label{positKernel}
\sum\limits_{i,j=1}^{N}
P_{\ell}^{(\frac{d}{2}-1+k, \frac{d}{2}-1)}(\langle\mathbf{x}_{i},\mathbf{x}_{j}\rangle)\geq 0,
\end{equation}
which follows from the fact that all coefficients in \eqref{ChristoffelDarbu1}
and \eqref{ChristoffelDarbu2} are positive and the fact that $P_{\ell}^{(d)}$
is a positive definite function in the sense of Schoenberg
\cite{Schoenberg1942PositiveDefinite}.

From (\ref{for112}) we obtain the following upper estimate
 \begin{multline}\label{f1}
   \frac{1}{N^{2}}\sum\limits_{i,j=1}^{N}\sum\limits_{\ell=n+1}^{\infty}
   a_\ell^{(0)}Z(d,\ell)
   P_{\ell}^{(d)}(\langle\mathbf{x}_{i},\mathbf{x}_{j}\rangle)\\
\leq \frac{1}{N^{2}}\sum\limits_{i,j=1}^{N}\sum\limits_{\ell=n+1}^{\infty}\Delta a_{\ell}^{(0)}\sum\limits_{k=0}^{\ell}Z(d,k)
 P_{k}^{(d)}(\langle\mathbf{x}_{i},\mathbf{x}_{j}\rangle).
\end{multline}

Taking into account (\ref{ChristoffelDarbu1}), applying   Abel transform and formulas (\ref{ChristoffelDarbu2}) and (\ref{alK}) $K-1$ times and using positive definiteness in
every step we arrive at
\begin{multline}\label{fK}
    \frac{1}{N^{2}}\sum\limits_{i,j=1}^{N}\sum\limits_{\ell=n+1}^{\infty}
   a_\ell^{(0)}Z(d,\ell)
   P_{\ell}^{(d)}(\langle\mathbf{x}_{i},\mathbf{x}_{j}\rangle)\\
   \ll
   \frac{1}{N^{2}}\sum\limits_{i,j=1}^{N}\sum\limits_{\ell=n+1}^{\infty}a_{\ell}^{(K)}  \frac{2\ell+d-1+K}{d-1+K}
\frac{(d+K-1)_{\ell}}{(d/2)_{\ell}}
 P_{\ell}^{(\frac{d}{2}+K-1, \
   \frac{d}{2}-1)}(\langle\mathbf{x}_{i},\mathbf{x}_{j}\rangle). 
\end{multline}

From formulas (\ref{Pochhammer}) and (\ref{gamma})
we get
\begin{equation}\label{gammaTh}
\frac{(d+K-1)_{\ell}}{(d/2)_{\ell}}=\frac{\Gamma\big(\frac{d}{2}\big)}{\Gamma(d+K-1)}
\frac{\Gamma(d+K-1+\ell)}{\Gamma\big(\frac{d}{2}+\ell\big)}\sim
\frac{\Gamma\big(\frac{d}{2}\big)}{\Gamma(d+K-1)}\ell^{\frac{d}{2}+K-1}.
 \end{equation}

 Relations (\ref{alK}), (\ref{f1})-(\ref{gammaTh}) yield
 (\ref{Lemma1}) and Lemma~\ref{lem1} is proved.
\end{proof}

\begin{proof}[Proof of Lemma~\ref{lem2}]
  The upper estimate in (\ref{Lemma2}) follows from (\ref{Lemma1}). Let us show
  that the lower estimate is true.

  Rewriting the squared worst-case error as above using $K$ times iterated Abel
  transform and formulas (\ref{ChristoffelDarbu1}), (\ref{ChristoffelDarbu2}), (\ref{alK}) and (\ref{positKernel}), we obtain

\begin{multline}\label{f118}
  \mathrm{wce}(Q[X_{N}];\mathbb{H}^{(\frac{d}{2},\gamma)}(\mathbb{S}^{d}))^2\\
\gg\frac{1}{N^{2}}\sum\limits_{i,j=1}^{N}\sum\limits_{\ell=t+1}^{\infty}a_{\ell}^{(K)}  \frac{2\ell+d-1+K}{d-1+K}
\frac{(d+K-1)_{\ell}}{(d/2)_{\ell}}
 P_{\ell}^{(\frac{d}{2}+K-1, \frac{d}{2}-1)}(\langle\mathbf{x}_{i},\mathbf{x}_{j}\rangle)\\
-\sum\limits_{m=0}^{K-1}a_{t+1}^{(m)}\frac{1}{N^{2}}\sum\limits_{i,j=1}^{N}
\frac{(d+m)_{t}}{(d/2)_{t}}
 P_{t}^{(\frac{d}{2}+m, \frac{d}{2}-1)}(\langle\mathbf{x}_{i},\mathbf{x}_{j}\rangle).
\end{multline}

Because of the exactness of the numerical integration rule for polynomials of
degree $\leq t$ and of (\ref{a3}), we have
\begin{multline}\label{f119}
\frac{1}{N}\sum\limits_{i=1}^{N}
 P_{t}^{(\frac{d}{2}+m, \frac{d}{2}-1)}(\langle\mathbf{x}_{i},\mathbf{x}_{j}\rangle)=\int\limits_{\mathbb{S}^{d}}P_{t}^{(\frac{d}{2}+m, \frac{d}{2}-1)}(\langle\mathbf{x}_{i},\mathbf{x}\rangle)d\sigma_{d}(\mathbf{x})\\
=2^{d-1}\frac{\Gamma(\frac{d+1}{2})}{\sqrt{\pi}}\frac{(m+1)_{t}}{t!}\frac{\Gamma(\frac{d}{2}+t)}{\Gamma(d+t)}
\asymp t^{m-\frac{d}{2}}.
\end{multline}

From (\ref{Pochhammer}) and (\ref{alK}) we obtain the order estimate
\begin{equation}\label{f120}
a_{t+1}^{(m)}
\frac{(d+m)_{t}}{(d/2)_{t}}
 t^{m-\frac{d}{2}}\asymp t^{-d-2m}(\ln t)^{-2\gamma} t^{m+\frac{d}{2}}t^{m-\frac{d}{2}}= t^{-d}(\ln t)^{-2\gamma}.
\end{equation}

Formulas (\ref{gammaTh}), (\ref{f118})--(\ref{f120}) imply that
\begin{multline}\label{lemma1LowerEstimate}
\mathrm{wce}(Q[X_{N}];\mathbb{H}^{(\frac{d}{2},\gamma)}(\mathbb{S}^{d}))^2\\
\gg\frac{1}{N^{2}}\sum\limits_{i,j=1}^{N}\sum\limits_{\ell=t+1}^{\infty} \ell^{-\frac{d}{2}-K}(\ln \ell)^{-2\gamma}
 P_{\ell}^{(\frac{d}{2}+K-1, \frac{d}{2}-1)}(\langle\mathbf{x}_{i},\mathbf{x}_{j}\rangle)
 -C_{d,\gamma} t^{-d}(\ln t)^{-2\gamma}.
\end{multline}
Thus, Lemma 2 is proved.
\end{proof}

\begin{proof}[Proof of Lemma 3]
  For each $i\in\{1,\ldots,N\}$ we divide the sphere $\mathbb{S}^{d}$ into an
  upper hemisphere $H_{i}^{+}$ with 'north pole' $\mathbf{x}_{i}$ and a lower
  hemisphere $H_{i}^{-}$:
\begin{align*}
H_{i}^{+}&:=\Big\{\mathbf{x}\in\mathbb{S}^{d}\Big|\langle\mathbf{x}_{i},\mathbf{x}\rangle\geq0 \Big\},\\
H_{i}^{-}&:=\mathbb{S}^{d}\setminus H_{i}^{+}.
\end{align*}
Because the spherical cap $S\big(\mathbf{x}_{i};\alpha_{N}\big)$, where
$\alpha_{N}:=\arccos\Big(1-\frac{c^{2}_{3}}{8N^{\frac{2}{d}}}\Big)$, contains
no points of $X_{N}$ in its interior, except of the point $\mathbf{x}_{i}$, we
obtain
\begin{multline}\label{aa}
\frac{1}{N^{2}}\sum\limits_{i,j=1}^{N}\sum\limits_{\ell=t+1}^{\infty}\ell^{-\frac{d}{2}-K}(\ln \ell)^{-2\gamma}
 P_{\ell}^{(\frac{d}{2}+K-1, \frac{d}{2}-1)}(\langle\mathbf{x}_{i},\mathbf{x}_{j}\rangle)\\
=\frac{1}{N^{2}}\sum\limits_{j=1}^{N}
{\mathop{\sum}\limits_{
 i=1,\atop \mathbf{x}_{i}\in H_{i}^{\pm}\setminus
S(\pm\mathbf{x}_{j}; \alpha_{N}) }^{N}} 
\sum\limits_{\ell=t+1}^{\infty}\ell^{-\frac{d}{2}-K}(\ln \ell)^{-2\gamma}
 P_{\ell}^{(\frac{d}{2}+K-1, \frac{d}{2}-1)}(\langle\mathbf{x}_{i},\mathbf{x}_{j}\rangle)\\
+\frac{1}{N^{2}}\sum\limits_{j=1}^{N}
{\mathop{\sum}\limits_{
 i=1,\atop \mathbf{x}_{i}\in 
S(-\mathbf{x}_{j}; \alpha_{N}) }^{N}} 
\sum\limits_{\ell=t+1}^{\infty}\ell^{-\frac{d}{2}-K}(\ln \ell)^{-2\gamma}
 P_{\ell}^{(\frac{d}{2}+K-1, \frac{d}{2}-1)}(\langle\mathbf{x}_{i},\mathbf{x}_{j}\rangle)\\
+\frac{1}{N}\sum\limits_{\ell=t+1}^{\infty}\ell^{-\frac{d}{2}-K}(\ln \ell)^{-2\gamma}
 P_{\ell}^{(\frac{d}{2}+K-1, \frac{d}{2}-1)}(1).
 \end{multline}
 
From (\ref{JacobiMax}) and the relation 
\begin{equation*}
\sum\limits_{j=n+1}^{\infty}\xi(j)=\int\limits_{n}^{\infty}\xi(u)
du+ \mathcal{O}(\xi(n)),
\end{equation*}
which holds for any positive and decreasing function $\xi(u)$, $u\geq 1$, such
that $\int\limits_{n}^{\infty}\xi(u) du<\infty$, we have
\begin{multline}\label{diagonal}
\frac{1}{N}\sum\limits_{\ell=t+1}^{\infty}\ell^{-\frac{d}{2}-K}(\ln \ell)^{-2\gamma}
 P_{\ell}^{(\frac{d}{2}+K-1, \frac{d}{2}-1)}(1)
 \sim
\frac{1}{\Gamma\big(\frac{d}{2}+K\big)}
 \frac{1}{N}\sum\limits_{\ell=t+1}^{\infty}\ell^{-1}(\ln \ell)^{-2\gamma}\\
 =C_{d,\gamma} \frac{1}{N} (\ln t)^{1-2\gamma}+O\Big(\frac{1}{N}t^{-1}(\ln t)^{-2\gamma}\Big).
 \end{multline}
 
Now we estimate the second term from the equality (\ref{aa}). 
An application of equality (\ref{JacobiMinus}) yields
\begin{multline}\label{afterMinus}
\frac{1}{N^{2}}\sum\limits_{j=1}^{N}
{\mathop{\sum}\limits_{
 i=1,\atop \mathbf{x}_{i}\in 
S(-\mathbf{x}_{j}; \alpha_{N}) }^{N}} 
\sum\limits_{\ell=t+1}^{\infty}\ell^{-\frac{d}{2}-K}(\ln \ell)^{-2\gamma}
 P_{\ell}^{(\frac{d}{2}+K-1, \frac{d}{2}-1)}(\langle\mathbf{x}_{i},\mathbf{x}_{j}\rangle)\\
=\frac{1}{N^{2}}\sum\limits_{j=1}^{N}
{\mathop{\sum}\limits_{
 i=1,\atop \mathbf{x}_{i}\in 
S(-\mathbf{x}_{j}; \alpha_{N}) }^{N}} 
\sum\limits_{\ell=t+1}^{\infty}(-1)^{\ell}\ell^{-\frac{d}{2}-K}(\ln \ell)^{-2\gamma}
 P_{\ell}^{(\frac{d}{2}-1, \frac{d}{2}+K-1)}(-\langle\mathbf{x}_{i},\mathbf{x}_{j}\rangle).
 \end{multline}
 
  If $\mathbf{x}_{i}\in S(-\mathbf{x}_{j}; \alpha_{N})$, then 
 \begin{equation}\label{scalarProd}
 -\langle\mathbf{x}_{i},\mathbf{x}_{j}\rangle \geq\cos\alpha_{N}.
  \end{equation}
 
 From the elementary estimates
\begin{equation*}
\sin\theta\leq \theta\leq \frac{\pi}{2}\sin\theta, \quad 0\leq\theta\leq \frac{\pi}{2},
\end{equation*}
we obtain
\begin{equation}\label{alpha}
\Big(1-\frac{c^{2}_{3}}{16N^{\frac{2}{d}}}\Big)^{\frac{1}{2}}\frac{c_{3}}{2N^{\frac{1}{d}}}\leq\alpha_{N}\leq 
\frac{\pi}{4}\Big(1-\frac{c^{2}_{3}}{16N^{\frac{2}{d}}}\Big)^{\frac{1}{2}}\frac{c_{3}}{N^{\frac{1}{d}}}.
 \end{equation} 
 
 As for the sequence $X_{N}$, condition (\ref{condiii}) holds, it means that
 the spherical cap $S(-\mathbf{x}_{j}; \alpha_{N})$, $j=1,\ldots,N$, contains at
 most one point of the sequence $X_{N}$. This fact and formulas
 (\ref{afterMinus})--(\ref{alpha}) imply
 \begin{multline}\label{afterMinus1}
\bigg|\frac{1}{N^{2}}\sum\limits_{j=1}^{N}
{\mathop{\sum}\limits_{
 i=1,\atop \mathbf{x}_{i}\in 
S(-\mathbf{x}_{j}; \alpha_{N}) }^{N}} 
\sum\limits_{\ell=t+1}^{\infty}\ell^{-\frac{d}{2}-K}(\ln \ell)^{-2\gamma}
 P_{\ell}^{(\frac{d}{2}+K-1, \frac{d}{2}-1)}(\langle\mathbf{x}_{i},\mathbf{x}_{j}\rangle)\bigg|\\
\leq\frac{1}{N}
\sum\limits_{\ell=t+1}^{\infty}\ell^{-\frac{d}{2}-K}(\ln \ell)^{-2\gamma}
 \Big|P_{\ell}^{(\frac{d}{2}-1, \frac{d}{2}+K-1)}(\cos\theta_{N})\Big|,
 \end{multline}
 for some $\theta_N$ satisfying
 \begin{equation}\label{thetaN}
0\leq\theta_{N}\leq \frac{\pi}{4}\Big(1-\frac{c^{2}_{3}}{16N^{\frac{2}{d}}}\Big)^{\frac{1}{2}}\frac{c_{3}}{N^{\frac{1}{d}}}.
 \end{equation}
 
 Let $\theta_{N}>0$ and $\ell^{*}\in\mathbb{N}$ is such that 
 \begin{equation*}
 \frac{1}{\ell^{*}+1}\leq\theta_{N}\leq\frac{1}{\ell^{*}},
 \end{equation*}
 and $\ell^{*}=\infty$, if $\theta_{N}=0$.
 
 Then, applying the estimates (\ref{JacobiIneq1}), (\ref{constantc0}) and
 (\ref{thetaN}), we get
 \begin{multline}\label{estimMinus}
 \frac{1}{N}
\sum\limits_{\ell=t+1}^{\infty}\ell^{-\frac{d}{2}-K}(\ln \ell)^{-2\gamma}
 \Big|P_{\ell}^{(\frac{d}{2}-1, \frac{d}{2}+K-1)}(\cos\theta_{N})\Big|
 \ll\frac{1}{N}
\sum\limits_{\ell=t+1}^{\ell^{*}}\ell^{-\frac{d}{2}-K}(\ln \ell)^{-2\gamma}
 \ell^{\frac{d}{2}-1}\\
+\frac{1}{N}\theta_{N}^{-\frac{1}{2}-\frac{d}{2}+1}
\sum\limits_{\ell=\ell^{*}+1}^{\infty}\ell^{-\frac{d}{2}-K}(\ln \ell)^{-2\gamma}
 \ell^{-\frac{1}{2}}\ll N^{-\frac{K}{d}-1}(\ln N)^{-2\gamma}.
 \end{multline}
 
Now let us show that 
\begin{equation}\label{reminder}
\bigg|\frac{1}{N^{2}}\sum\limits_{j=1}^{N}
{\mathop{\sum}\limits_{
 i=1,\atop \mathbf{x}_{i}\in H_{i}^{\pm}\setminus
S(\pm\mathbf{x}_{j}; \alpha_{N}) }^{N}} 
\sum\limits_{\ell=t+1}^{\infty}\ell^{-\frac{d}{2}-K}(\ln \ell)^{-2\gamma}
 P_{\ell}^{(\frac{d}{2}+K-1, \frac{d}{2}-1)}(\langle\mathbf{x}_{i},\mathbf{x}_{j}\rangle)\bigg|
\ll \frac{1}{N} (\ln t)^{-2\gamma}.
 \end{equation}

Using formula (\ref{JacobiIneq}), we have that for $0<\theta<\pi$
\begin{equation}\label{JacobiD}
|P_{\ell}^{(\frac{d}{2}+K-1, \frac{d}{2}-1)}(\cos \theta)|\ll\ell^{-\frac{1}{2}}(\sin\theta)^{-\frac{d}{2}-K+\frac{1}{2}}+
 \ell^{-\frac{3}{2}}(\sin\theta)^{-\frac{d}{2}-K-\frac{1}{2}}.
 \end{equation}
Then
\begin{multline*}
\Big|\sum\limits_{\ell=t+1}^{\infty}\ell^{-\frac{d}{2}-K}(\ln \ell)^{-2\gamma}
 P_{\ell}^{(\frac{d}{2}+K-1, \frac{d}{2}-1)}(\cos\theta)\Big|\\
\ll(\sin\theta)^{-\frac{d}{2}-K+\frac{1}{2}}\sum\limits_{\ell=t+1}^{\infty}
\ell^{-\frac{d}{2}-K-\frac{1}{2}}(\ln \ell)^{-2\gamma}\\
+
(\sin\theta)^{-\frac{d}{2}-K-\frac{1}{2}}\sum\limits_{\ell=t+1}^{\infty}\ell^{-\frac{d}{2}-K-\frac{3}{2}}(\ln \ell)^{-2\gamma}\\
\ll(\sin\theta)^{-\frac{d}{2}-K+\frac{1}{2}}
t^{-\frac{d}{2}-K+\frac{1}{2}}(\ln t)^{-2\gamma}+
(\sin\theta)^{-\frac{d}{2}-K-\frac{1}{2}}t^{-\frac{d}{2}-K-\frac{1}{2}}(\ln t)^{-2\gamma}.
\end{multline*}

We define $\theta_{ij}^{\pm}\in[0,\pi]$ by
$\cos \theta_{ij}^{\pm}:=\langle\mathbf{x}_{i},\pm\mathbf{x}_{j}\rangle$. Then
$\sin \theta_{ij}^{+}=\sin \theta_{ij}^{-}$.

So,
\begin{align}
\bigg|\frac{1}{N^{2}}\sum\limits_{j=1}^{N}
{\mathop{\sum}\limits_{
 i=1,\atop \mathbf{x}_{i}\in H_{i}^{\pm}\setminus
S(\pm\mathbf{x}_{j}; \alpha_{N}) }^{N}} 
\sum\limits_{\ell=t+1}^{\infty}\ell^{-\frac{d}{2}-K}(\ln \ell)^{-2\gamma}
 P_{\ell}^{(\frac{d}{2}+K-1, \frac{d}{2}-1)}(\langle\mathbf{x}_{i},\mathbf{x}_{j}\rangle)\bigg| \notag \\
\ll
t^{-\frac{d}{2}-K+\frac{1}{2}}(\ln t)^{-2\gamma}
\frac{1}{N^{2}}\sum\limits_{j=1}^{N}
{\mathop{\sum}\limits_{
 i=1,\atop \mathbf{x}_{i}\in H_{i}^{\pm}\setminus
S(\pm\mathbf{x}_{j}; \alpha_{N}) }^{N}} 
(\sin\theta_{ij}^{\pm})^{-\frac{d}{2}-K+\frac{1}{2}}\notag  \\ \label{for10}
 +t^{-\frac{d}{2}-K-\frac{1}{2}}(\ln t)^{-2\gamma}
\frac{1}{N^{2}}\sum\limits_{j=1}^{N}{\mathop{\sum}\limits_{
 i=1,\atop \mathbf{x}_{i}\in H_{i}^{\pm}\setminus
S(\pm\mathbf{x}_{j}; \alpha_{N}) }^{N}}(\sin\theta_{ij}^{\pm})^{-\frac{d}{2}-K-\frac{1}{2}}.
 \end{align}

 From \cite[(3.30) and (3.33)]{Brauchart-Hesse2007:numerical_integration}, it
 follows that
 \begin{align}
  \frac{1}{N^{2}}\sum\limits_{j=1}^{N}{\mathop{\sum}\limits^{N}_{
 i=1,\atop \mathbf{x}_{i}\in H_{j}^{\pm}\setminus
S(\pm\mathbf{x}_{j}; \frac{c_{1}}{n})}}
  (\sin\theta_{ij}^{\pm})^{-\frac{d}{2}+\frac{1}{2}-k-L}
 \notag \\ \label{BrauchartHesse}
 \ll 1+n^{L+k-(d+1)/2}, \quad k=0,1,\ldots \quad \text{for }L>\frac{d+1}{2}.
 \end{align}
Choosing $K>\frac{d+1}{2}$, applying (\ref{constantc0}) and (\ref{BrauchartHesse}), we obtain
\begin{equation}\label{for11}
\frac{1}{N^{2}}\sum\limits_{j=1}^{N}{\mathop{\sum}\limits_{
 i=1,\atop \mathbf{x}_{i}\in H_{j}^{\pm}\setminus
S(\pm\mathbf{x}_{j}; \alpha_{N}) }^{N}}(\sin\theta_{ij}^{\pm})^{-\frac{d}{2}-K\pm\frac{1}{2}}
 \ll 1+(N^{\frac{1}{d}})^{K-\frac{d}{2}\mp\frac{1}{2}}\ll(N^{\frac{1}{d}})^{K-\frac{d}{2}\mp\frac{1}{2}},
 \end{equation}
  Formulas (\ref{constantc0}), (\ref{for10}) and (\ref{for11})  now imply that estimate (\ref{reminder}) holds.
 
 From (\ref{constantc0}), (\ref{diagonal}), (\ref{estimMinus}) and
 (\ref{reminder}) we obtain (\ref{Lemma3}) and
 Lemma 3 is proved.
 \end{proof}

\section{Lower bounds for the  worst-case error}
\label{lower}

The main result of this section is the following theorem.

\begin{thm}\label{thm4} Let $d\geq2$, $\gamma>\frac{1}{2}$, $Q[X_{N},\omega]$
  is an arbitrary  $N$-point cubature rule. Then, there exists a positive
  constant $C_{d,\gamma}$ such that
\begin{equation}\label{theoremLowerBound}
  \mathrm{wce}(Q[X_{N},\omega];\mathbb{H}^{(\frac{d}{2},\gamma)}(\mathbb{S}^{d}))
  \geq C_{d,\gamma}N^{-\frac{1}{2}}\left(\ln N\right)^{-\gamma}.
 \end{equation}
 The constant $C_{d,\gamma}$ depends only on $d$ and $\gamma$, but is
 independent of the rule $Q[X_{N},\omega]$ and the number of nodes $N$ of the
 rule.
 \end{thm}

 In \cite{Hesse-Sloan2006:cubature_sobolev} for case $d=2$ and in
 \cite{Hesse2006:lower_cubature} for all $d\geq2$ the lower bound
\begin{equation}\label{lowerBoundWceHs}
\mathrm{wce}(Q[X_{N},\omega];\mathbb{H}^{s}(\mathbb{S}^{d}))\gg N^{-\frac{s}{d}}
\end{equation}
 was found.

 Before we actually give the proof of Theorem~\ref{thm4}, we formulate a
 packing result \cite[Lemma 1]{Hesse2006:lower_cubature}.

 \begin{stat}\label{stat1} Let $d\geq2$. Then there exist constants $\tilde{c}_{1}>0$ and $\tilde{c}_{2}\geq 1$
depending only on $d$, such that for any $N\in\mathbb{N}$, there exist
   $N_{0}$
   points $\mathbf{y}_{1},\ldots,\mathbf{y}_{N_{0}}$ on $\mathbb{S}^{d}$ and an
   angle $\beta_{N}$, with
\begin{align*}
\beta_{N}&=\tilde{c}_{1}(2N)^{-\frac{1}{d}},\\
2N&\leq N_{0}\leq \tilde{c}_{2}2N,
\end{align*}
such that the caps $S(\mathbf{y}_{i}; \beta_{N})$, $i=1,\ldots,N_{0}$ form a
packing of $\mathbb{S}^{d}$ (that is $S(\mathbf{y}_{i}; \beta_{N})$ and
$S(\mathbf{y}_{j}; \beta_{N})$ with $i\neq j$ touch at most at their
boundaries).
\end{stat}

As we consider a packing with $2N\geq 2$ caps in Statement~\ref{stat1}, the
angle $\beta_{N}$ can be at most $\frac{\pi}{2}$ (which is achieved for a
packing with 2 caps with opposite centres).

\begin{proof}[Proof of Theorem~\ref{thm4}] To prove the lower bound we will use
  the same 'fooling' function as in \cite{Hesse2006:lower_cubature}, that is a
  function which vanishes in all nodes of the cubature rule $Q[X_{N},\omega]$
  but has large integral and small norm.

At the beginning we construct the function $\Phi\in \mathbb{C}^{\infty}(\mathbb{R})$ with the following properties: 
(i) $\Phi(t)\geq 0$ for all $t\in\mathbb{R}$; 
(ii) $\max\limits_{t\in\mathbb{R}}\Phi(t)=\Phi(0)=1$;
(iii) $\Phi$ has the compact support $\mathrm{supp}(\Phi)=[-1, 1]$.

Statement~\ref{stat1} guarantees that there exist at least $2N$ spherical caps
$S(\mathbf{y}_{i}; \beta_{N})$, which touch at most at their
boundaries. Consequently, at least $N$ of these spherical caps do not contain
any node of the cubature rule in their interior.

We shift the argument of the function $\Phi$ in such a way, that the support of
the function will be $[\cos\beta_{N}, \cos\frac{\beta_{N}}{2}]$.
 
 The scaled version of $\Phi$ is given by
 \begin{equation*}
 \Phi_{N}(t):=\Phi\bigg(\frac{2t-(\cos\frac{\beta_{N}}{2}+\cos\beta_{N})}{2\sin\frac{3\beta_{N}}{4} \sin\frac{\beta_{N}}{4}} \bigg), \quad t\in\mathbb{R}.
 \end{equation*}

We define our 'fooling' function 
$f_{N}\in \mathbb{C}^{\infty}(\mathbb{S}^{d})$ by
\begin{equation*}
f_{N}(\mathbf{x}):=\sum\limits_{i=1}^{N}\Phi_{N}(\langle\mathbf{x},\mathbf{y}_{i}\rangle), \quad \mathbf{x}\in \mathbb{S}^{d}.
\end{equation*}

In \cite{Hesse2006:lower_cubature} it was proved that for all $s\geq0$
 \begin{equation}\label{normPhiHs}
\|f_{N}\|_{\mathbb{H}^{s}}\leq C_{s,d}N^{\frac{s}{d}}.
 \end{equation}

 The function $f_{N}$ vanishes in all nodes of the cubature rule, that is,
 $Q[X_{N},\omega](f_{N})=0$. And (see formula (33) of
 \cite{Hesse2006:lower_cubature})
 \begin{equation*}
 I(f_{N})\geq c_{d}.
 \end{equation*}
Hence,
\begin{multline}\label{Q-Int}
 \mathrm{wce}(Q[X_{N},\omega];\mathbb{H}^{(\frac{d}{2},\gamma)}(\mathbb{S}^{d}))
\geq\bigg|Q[X_{N},\omega]\Big(\frac{f_{N}}{\|f_{N}\|_{\mathbb{H}^{(\frac{d}{2},\gamma)}}}\Big)-I\Big(\frac{f_{N}}{\|f_{N}\|_{\mathbb{H}^{(\frac{d}{2},\gamma)}}}\Big) \bigg|\\
=\frac{I(f_{N})}{\|f_{N}\|_{\mathbb{H}^{(\frac{d}{2},\gamma)}}}\gg \frac{1}{\|f_{N}\|_{\mathbb{H}^{(\frac{d}{2},\gamma)}}}.
 \end{multline}
 
 The function $\Phi_{N}$ can be expanded on $[-1, 1]$ into an $L_{2}([-1,
 1])$ convergent Laplace series 
 \begin{equation*}
 \Phi_{N}=\sum\limits_{\ell=0}^{\infty}Z(d,\ell)\Big(\int\limits_{-1}^{1}\Phi_{N}(t)P_{\ell}^{(d)}(t)dt \Big)P_{\ell}^{(d)}.
 \end{equation*}
 Hence,
 \begin{equation}\label{fNLegandre}
  f_{N}(\mathbf{x})=\sum\limits_{i=1}^{N}\sum\limits_{\ell=0}^{\infty}Z(d,\ell)\bigg(\int\limits_{-1}^{1}\Phi_{N}(t)P_{\ell}^{(d)}(t)dt \bigg)P_{\ell}^{(d)}(\langle\mathbf{x},\mathbf{y}_{i}\rangle).
 \end{equation}
 
 Due to the definition (\ref{norm}), representation (\ref{fNLegandre}), the
 addition theorem (\ref{additiontheorem}) and inequality (\ref{normPhiHs}), we
 have the following estimate
 \begin{equation}\label{normHs}
 \|f_{N}\|^{2}_{\mathbb{H}^{s}}
= \sum\limits_{\ell=0}^{\infty}\bigg(\int\limits_{-1}^{1}\Phi_{N}(t)P_{\ell}^{(d)}(t)dt \bigg)^{2}
\left(1+\lambda_{\ell}\right)^{s}
Z(d,\ell)\sum\limits_{i,j=1}^{N} P_{\ell}^{(d)}(\langle\mathbf{y}_{i},\mathbf{y}_{j}\rangle)\ll N^{\frac{2s}{d}},
 \end{equation} 
which holds for $s>0$ by \cite{Brauchart-Hesse2007:numerical_integration}.
 
  The corresponding norm of the function $f_{N}$ in $\mathbb{H}^{(\frac{d}{2},\gamma)}$ has the form
 \begin{multline}\label{normH}
 \|f_{N}\|^{2}_{\mathbb{H}^{(\frac{d}{2},\gamma)}}
=  \sum\limits_{\ell=0}^{[N^{\frac{1}{d}}]}\bigg(\int\limits_{-1}^{1}\Phi_{N}(t)P_{\ell}^{(d)}(t)dt \bigg)^{2}
w_{\ell}(d,\gamma)
Z(d,\ell)\sum\limits_{i,j=1}^{N} P_{\ell}^{(d)}(\langle\mathbf{y}_{i},\mathbf{y}_{j}\rangle) \\
+ \sum\limits_{\ell=[N^{\frac{1}{d}}]+1}^{\infty}\bigg(\int\limits_{-1}^{1}\Phi_{N}(t)P_{\ell}^{(d)}(t)dt \bigg)^{2}
w_{\ell}(d,\gamma)
Z(d,\ell)\sum\limits_{i,j=1}^{N} P_{\ell}^{(d)}(\langle\mathbf{y}_{i},\mathbf{y}_{j}\rangle).
 \end{multline}

 Taking into account (\ref{positKernel}) and setting $s=1$ in (\ref{normHs}), we
 obtain
\begin{equation}\label{normHsEstimateS=1}
\sum\limits_{\ell=0}^{[N^{\frac{1}{d}}]}\bigg(\int\limits_{-1}^{1}\Phi_{N}(t)P_{\ell}^{(d)}(t)dt \bigg)^{2}
(1+\lambda_{\ell})
Z(d,\ell)\sum\limits_{i,j=1}^{N} P_{\ell}^{(d)}(\langle\mathbf{y}_{i},\mathbf{y}_{j}\rangle)\ll N^{\frac{2}{d}}.
 \end{equation}
 
 Thus, (\ref{normHsEstimateS=1}) yields
 \begin{multline}\label{ineq1}
\sum\limits_{\ell=0}^{[N^{\frac{1}{d}}]}\bigg(\int\limits_{-1}^{1}\Phi_{N}(t)P_{\ell}^{(d)}(t)dt \bigg)^{2}
w_{\ell}(d,\gamma)
Z(d,\ell)\sum\limits_{i,j=1}^{N} P_{\ell}^{(d)}(\langle\mathbf{y}_{i},\mathbf{y}_{j}\rangle)\\
\ll(N^{\frac{1}{d}})^{d-2} (\ln N)^{2\gamma} \sum\limits_{\ell=0}^{[N^{\frac{1}{d}}]}\bigg(\int\limits_{-1}^{1}\Phi_{N}(t)P_{\ell}^{(d)}(t)dt \bigg)^{2}
(1+\lambda_{\ell})
Z(d,\ell)\sum\limits_{i,j=1}^{N} P_{\ell}^{(d)}(\langle\mathbf{y}_{i},\mathbf{y}_{j}\rangle)\\
\ll N (\ln N)^{2\gamma}.
 \end{multline}

Setting $s=\frac{d+1}{2}$ in (\ref{normHs}), we get
\begin{equation}\label{normHsEstimateS=d/2+1}
\sum\limits_{\ell=[N^{\frac{1}{d}}]+1}^{\infty}\bigg(\int\limits_{-1}^{1}\Phi_{N}(t)P_{\ell}^{(d)}(t)dt \bigg)^{2}
(1+\lambda_{\ell})^{\frac{d+1}{2}}
Z(d,\ell)\sum\limits_{i,j=1}^{N} P_{\ell}^{(d)}(\langle\mathbf{y}_{i},\mathbf{y}_{j}\rangle)\ll N^{1+\frac{1}{d}}.
 \end{equation}
 Thus, (\ref{normHsEstimateS=d/2+1}) yields
 \begin{multline}\label{ineq2}
  \sum\limits_{\ell=[N^{\frac{1}{d}}]+1}^{\infty}\bigg(\int\limits_{-1}^{1}\Phi_{N}(t)P_{\ell}^{(d)}(t)dt \bigg)^{2}
w_{\ell}(d,\gamma)
Z(d,\ell)\sum\limits_{i,j=1}^{N} P_{\ell}^{(d)}(\langle\mathbf{y}_{i},\mathbf{y}_{j}\rangle)\\
 \ll N^{-\frac{1}{d}}(\ln N)^{2\gamma} \sum\limits_{\ell=[N^{\frac{1}{d}}]+1}^{\infty}\bigg(\int\limits_{-1}^{1}\Phi_{N}(t)P_{\ell}^{(d)}(t)dt \bigg)^{2}
(1+\lambda_{\ell})^{\frac{d+1}{2}}
Z(d,\ell)\sum\limits_{i,j=1}^{N} P_{\ell}^{(d)}(\langle\mathbf{y}_{i},\mathbf{y}_{j}\rangle)\\
\ll N (\ln N)^{2\gamma}.
 \end{multline}

Estimates (\ref{normH}), (\ref{ineq1})  and (\ref{ineq2}) imply
 \begin{equation}\label{normEstimate}
\|f_{N}\|_{\mathbb{H}^{(\frac{d}{2},\gamma)}}\ll N^{\frac{1}{2}} (\ln N)^{\gamma}.
 \end{equation}

From (\ref{Q-Int}) and (\ref{normEstimate}) we obtain (\ref{theoremLowerBound})
and Theorem~\ref{thm4} is proved.
\end{proof}

We should remark,  that we
can obtain  Theorem~\ref{thm4} in the case of equal
weights by simply applying  \cite[Theorem~4.2]{Bilyk-Dai2016:geodesic_riesz}.

Let the zonal function $F$:
$F(\mathbf{x},\mathbf{y})=F(\langle\mathbf{x},\mathbf{y}\rangle)$,
$\mathbf{x},\mathbf{y}\in\mathbb{S}^{d}$ be continuous on the segment $[-1, 1]$
and have the form
\begin{equation}\label{expansF}
F(\mathbf{x},\mathbf{y})=\sum\limits_{\ell=0}^{\infty}\hat{F}(d,\ell)Z(d,\ell)
 P_{\ell}^{(d)}(\langle\mathbf{x},\mathbf{y}\rangle),
\end{equation}
where $\hat{F}(d,\ell)\geq0$.  

%

The following Statement 2 is \cite[Theorem~4.2]{Bilyk-Dai2016:geodesic_riesz}.
 
\begin{stat}\label{stat2} Let $\lambda=\dfrac{d-1}{2}$. Assume that $F$
  satisfies relation (\ref{expansF}).  Then there exists positive constants
  $c_d$ and $C_d$ depending only on $d$ and $F$, such that for any
  $N\in\mathbb{N}$ and a given set of
  $N$ points $X_{N}=\{\mathbf{x}_{1},.., \mathbf{x}_{N}\}\subset\mathbb{S}^{d}$
the inequality
  \begin{multline}\label{theoremDB}
   C_{d}\min\limits_{1\leq \ell\leq c_{d}N^{1/d}}\hat{F}(d,\ell)
   \leq
   \frac{1}{N^{2}}\sum\limits_{j=1}^{N}\sum\limits_{i=1}^{N}F(\langle\mathbf{x}_{i},\mathbf{x}_{j}\rangle)-\hat F(d,0)
 \end{multline}
 holds.
\end{stat}
  
Applying this statement to $F=\tilde K_{d,\gamma}$ gives
\begin{multline}\label{form}
\mathrm{wce}(Q[X_{N}];\mathbb{H}^{(\frac{d}{2},\gamma)}(\mathbb{S}^{d}))^{2}
=\sum\limits_{i,j=1}^{N}
\sum\limits_{\ell=1}^{\infty}w^{-1}_{\ell}(d,\gamma) Z(d,\ell)
 P_{\ell}^{(d)}(\langle\mathbf{x}_{i},\mathbf{x}_{j}\rangle)\\
\geq C_{d}\min\limits_{1\leq \ell\leq c_{d}N^{1/d}}w^{-2}_{\ell}(d,\gamma) 
  \gg C_{d}N^{-1}\left(\ln N\right)^{-2\gamma}.
 \end{multline}
And, therefore,
\begin{equation}\label{lowerestimate}
  \mathrm{wce}(Q[X_{N}];\mathbb{H}^{(\frac{d}{2},\gamma)}(\mathbb{S}^{d}))\geq C_{d,\gamma}N^{-1/2}\left(\ln N\right)^{-\gamma}.
 \end{equation}

In the same way, by applying (\ref{theoremDB}), one can easily obtain estimate (\ref{lowerBoundWceHs})  in the case of equal weights.
  
 \section{QMC designs for $\mathbb{H}^{(\frac{d}{2},\gamma)}(\mathbb{S}^{d})$ and their properties}
 \label{QMC}  
 
\subsection{QMC designs for $\mathbb{H}^{s}(\mathbb{S}^{d})$ and $\mathbb{H}^{(\frac{d}{2},\gamma)}(\mathbb{S}^{d})$ }

 Let us formulate at the beginning the definition of QMC-designs for Sobolev
 spaces $\mathbb{H}^{s}(\mathbb{S}^{d})$ (see,
 e.g. \cite{Brauchart-Saff-Sloan+2014:qmc_designs}).

 \begin{defi} Given $s>\frac{d}{2}$, a sequence $X_{N}$ of $N$--point
   configurations on $\mathbb{S}^{d}$ with $N\rightarrow\infty$ is said to be a
   sequence of QMC designs for $\mathbb{H}^{s}(\mathbb{S}^{d})$ if there exists
   $c(s,d)>0$, independent of $N$, such that
\begin{equation}\label{QmcHsDefinition}
\sup
\limits_{
 f\in \mathbb{H}^{s}, \|f\|_{\mathbb{H}^{s}}\leq1 }
 \left|\frac{1}{N}\sum\limits_{\mathbf{x}\in X_{N}}f(\mathbf{x})-\int_{\mathbb{S}^{d}}f(\mathbf{x})d\sigma_{d}(\mathbf{x}) \right|\leq\frac{c(s,d)}{N^{\frac{s}{d}}}.
\end{equation}
\end{defi}

We define the notion of a sequence of QMC designs for
$\mathbb{H}^{(\frac{d}{2},\gamma)}(\mathbb{S}^{d})$, $\gamma>\frac{1}{2}$, as
it was defined for Sobolev classes $\mathbb{H}^{s}(\mathbb{S}^{d})$,
$s>\frac{d}{2}$.
 
\begin{defi} Given $\gamma>\frac{1}{2}$, a sequence $(X_{N})_N$ of $N$--point
  configurations on $\mathbb{S}^{d}$ with $N\rightarrow\infty$ is said to be a
  sequence of QMC designs for
  $\mathbb{H}^{(\frac{d}{2},\gamma)}(\mathbb{S}^{d})$ if there exists
  $c(\gamma,d)>0$, independent of $N$, such that
\begin{equation}\label{QMCdefinition}
\sup
\limits_{
 f\in \mathbb{H}^{(\frac{d}{2},\gamma)}, \|f\|_{\mathbb{H}^{(\frac{d}{2},\gamma)}}\leq1 }
 \left|\frac{1}{N}\sum\limits_{\mathbf{x}\in X_{N}}f(\mathbf{x})-\int_{\mathbb{S}^{d}}f(\mathbf{x})d\sigma_{d}(\mathbf{x}) \right|\leq\frac{c(\gamma,d)}{N^{\frac{1}{2}}(\ln N)^{\gamma-\frac{1}{2}}}.
\end{equation}
\end{defi}

\begin{thm}\label{thm5} Given $s>\frac{d}{2}$, let $(X_{N})_N$ be a sequence of
  QMC designs for
  $\mathbb{H}^{s}(\mathbb{S}^{d})$. Then $(X_{N})_{N}$ is a sequence of QMC designs
  for $\mathbb{H}^{(\frac{d}{2},\gamma)}(\mathbb{S}^{d})$, for all
  $\gamma>\frac{1}{2}$.
  \end{thm}

\begin{thm}\label{thm8} Given $\gamma>\frac{1}{2}$, let $(X_{N})_{N}$ be a sequence
  of QMC designs for $\mathbb{H}^{(\frac{d}{2},\gamma)}$. Then $(X_{N})_{N}$ is a
  sequence of QMC designs for
  $\mathbb{H}^{(\frac{d}{2},\gamma')}(\mathbb{S}^{d})$, for all
  ${\frac{1}{2}<\gamma'\leq\gamma}$.
  \end{thm}

We will prove here only Theorem~\ref{thm5}. The proof of Theorem~\ref{thm8} follows the same lines as  that of Theorem~\ref{thm5} with some additional estimations.

Proof of Theorem~\ref{thm5} is based on the following lemma.

 \begin{lemma}\label{lem4} Assume that there exists a $\delta>0$, such that 
 \begin{align}\label{LemmaIneq}
\mathrm{wce}(Q[X_{N}];\mathbb{H}^{s}(\mathbb{S}^{d}))\ll N^{-\delta},
\end{align}
  holds for some $s>\frac{d}{2}$. Then for $\gamma>\frac{1}{2}$ there exists a
  constant $C(d,s,\delta, \gamma)$ such that for all $N$
 \begin{align}\label{Lemma4}
\mathrm{wce}(Q[X_{N}];\mathbb{H}^{(\frac{d}{2},\gamma)}(\mathbb{S}^{d}))<C(d,s,\delta,\gamma)
[\mathrm{wce}(Q[X_{N}];\mathbb{H}^{s}(\mathbb{S}^{d}))]^{\frac{d}{2s}}(\ln N)^{-\gamma+\frac{1}{2}}
\end{align}
holds.
\end{lemma}

\begin{proof}[Proof of Lemma~\ref{lem4}] 
  The proof of (\ref{Lemma4}) goes along the lines as that of Lemma 26 in
  \cite{Brauchart-Saff-Sloan+2014:qmc_designs} and Theorem 3.1 in
  \cite{BrandChoirColzGigSeriTravQuadrature}.

We write 
\begin{align}\label{LaplaceTransform}
\frac{1}{(1+\lambda_{\ell})^{\frac{d}{2}}(\ln(3+\lambda_{\ell}))^{2\gamma}}=\int\limits_{0}^{\infty}e^{-(1+\lambda_{\ell})t}g(t)dt,
\end{align}
in terms of the Laplace transform of the function $g$ given by the inverse
Laplace transform
\begin{align}\label{InverseLaplaceTransform}
g(t)=g(d,\gamma,t):=\frac{1}{2\pi i}\int\limits_{\frac{1}{t}-i\infty}^{\frac{1}{t}+i\infty}
z^{-\frac{d}{2}}(\ln(z+2))^{-2\gamma}e^{tz}dz.
\end{align}

First of all, let us show, that the function $g$ satisfies
\begin{align}\label{estimation_g}
|g(t)|\ll\begin{cases}
 t^{\frac{d}{2}-1}, & \text{if }
 t\geq1, \\
 t^{\frac{d}{2}-1} (\ln \frac{1}{t})^{-2\gamma}, & \text{if }  0 <t<1.
  \end{cases}
\end{align}

Indeed, substituting $tz=1+ix$ and integrating by parts, we obtain
 \begin{multline}\label{estimation_g1}
\frac{1}{2\pi i}\int\limits_{\frac{1}{t}-i\infty}^{\frac{1}{t}+i\infty}
z^{-\frac{d}{2}}(\ln(z+2))^{-2\gamma}e^{tz}dz\\
=\frac{e}{2\pi }t^{\frac{d}{2}-1}\int\limits_{-\infty}^{\infty}
(1+ix)^{-\frac{d}{2}}\Big(\ln\Big(2+\frac{1+ix}{t}\Big)\Big)^{-2\gamma}e^{ix}dx\\
=\frac{e}{2\pi }t^{\frac{d}{2}-1}\int\limits_{-\infty}^{\infty}e^{ix}
\Big[\frac{d}{2}(1+ix)^{-\frac{d}{2}-1}\Big(\ln\Big(2+\frac{1+ix}{t}\Big)\Big)^{-2\gamma}\\
+2\gamma(1+ix)^{-\frac{d}{2}}(2t+1+ix)^{-1}\Big(\ln\Big(2+\frac{1+ix}{t}\Big)\Big)^{-2\gamma-1}\Big]dx.
\end{multline}

For large values of $t: t\geq1$ from (\ref{estimation_g1}) one can easily get
$|g(t)|\ll t^{\frac{d}{2}-1}$.

In turn, for small values of $t: 0<t<1$, the inequalities
\begin{align*}
\Big|\ln\Big(2+\frac{1+ix}{t}\Big) \Big|>\ln \frac{1}{t}, \ \ \ 
\Big(\ln \frac{1}{t}\Big)^{-2\gamma-1}<\Big(\ln \frac{1}{t}\Big)^{-2\gamma}, \ \ 0<t<1,
\end{align*}
and relation (\ref{estimation_g1}) imply that 
$|g(t)|\ll t^{\frac{d}{2}-1} (\ln \frac{1}{t})^{-2\gamma}$.

The representation of the worst-case error (\ref{wceKernel}) allows to write
\begin{align}\label{wceLaplaceH}
\mathrm{wce}(Q[X_{N}];\mathbb{H}^{(\frac{d}{2},\gamma)}(\mathbb{S}^{d}))^{2}=\int\limits_{0}^{\infty}e^{-t}g(t)h(t)dt,
\end{align}
where
\begin{align}\label{function_h}
h(t)=h(t;\mathbf{x}_{1},...,\mathbf{x}_{N}):=\frac{1}{N^{2}}\sum\limits_{i,j=1}^{N}\tilde{H}(t,\langle\mathbf{x}_{i},\mathbf{x}_{j}\rangle),
\end{align}
and $\tilde{H}$ denotes the heat kernel with the constant term removed:
\begin{align}\label{heatKernel}
  1+\tilde H(t,\mathbf{x},\mathbf{y}):=
  \sum\limits_{\ell=0}^{\infty}e^{-\lambda_{\ell}t}Z(d,\ell)
 P_{\ell}^{(d)}(\langle\mathbf{x},\mathbf{y}\rangle), \ \ \mathbf{x},\mathbf{y}\in \mathbb{S}^{d},
\end{align}
which is fundamental solution to the heat equation $\frac{\partial u}{\partial t}+\Delta^{*}_{d}u=0$ on $\mathbb{R}_{+}\times \mathbb{S}^{d}$.

The worst-case error for Sobolev spaces in terms of Laplace transform can be
written in the form (see formula (46) in
\cite{Brauchart-Saff-Sloan+2014:qmc_designs})
\begin{align}\label{wceLaplaceHs}
\mathrm{wce}(Q[X_{N}];\mathbb{H}^{s}(\mathbb{S}^{d}))^{2}=\frac{1}{\Gamma(s)}\int\limits_{0}^{\infty}e^{-t}t^{s-1}h(t)dt.
\end{align}

Let $\varepsilon:=[\mathrm{wce}(Q[X_{N}];\mathbb{H}^{s}(\mathbb{S}^{d}))]^{\frac{2}{s}}$, and $\varepsilon\ll N^{-\delta}<1$ by assumption.

The first inequality from (\ref{estimation_g}) and (\ref{wceLaplaceHs}) yield
\begin{align}\label{large_t}
\bigg|\int\limits_{1}^{\infty}e^{-t}g(t)h(t)dt \bigg|\ll \int\limits_{1}^{\infty}e^{-t}t^{\frac{d}{2}-1}h(t)dt \ll \frac{1}{\Gamma(s)}\int\limits_{0}^{\infty}e^{-t}t^{s-1}h(t)dt=\varepsilon^{s}, \ \ s>\frac{d}{2}.
 \end{align}
 Taking into account the second inequality from (\ref{estimation_g}),
 (\ref{LemmaIneq}) and (\ref{wceLaplaceHs}), we get
\begin{multline}\label{medium_t}
\bigg|\int\limits_{\frac{\varepsilon}{2}}^{1}e^{-t}g(t)h(t)dt \bigg|\ll 
\int\limits_{\frac{\varepsilon}{2}}^{1}e^{-t}t^{\frac{d}{2}-1} \Big(\ln \frac{1}{t}\Big)^{-2\gamma}h(t)dt 
\\
\leq \Big(\frac{\varepsilon}{2}\Big)^{\frac{d}{2}-s}\Big( \ln\Big( \frac{2}{\varepsilon}\Big) \Big)^{-2\gamma}\int\limits_{\frac{\varepsilon}{2}}^{1}e^{-t}t^{s-1} h(t)dt 
\\
\ll \varepsilon^{\frac{d}{2}-s}(\ln N)^{-2\gamma}\frac{1}{\Gamma(s)}\int\limits_{0}^{\infty}e^{-t}t^{s-1} h(t)dt =\varepsilon^{\frac{d}{2}}(\ln N)^{-2\gamma}.
 \end{multline}
 
 In \cite{Brauchart-Saff-Sloan+2014:qmc_designs} it was proved, that $h(t)$ is
 uniformly bounded on $[0,1)$, and for $0<t<\frac{\varepsilon}{2}$ the
 following etimate holds
 \begin{align}\label{uniformBound_h}
t^{\frac{d}{2}}h(t)\ll \varepsilon^{\frac{d}{2}}.
 \end{align}
 
 Applying (\ref{estimation_g}), relations (\ref{LemmaIneq}) and
 (\ref{uniformBound_h}), we arrive at the estimate
 \begin{multline}\label{small_t}
\bigg|\int\limits_{0}^{\frac{\varepsilon}{2}}e^{-t}g(t)h(t)dt \bigg|\ll 
\int\limits_{0}^{\frac{\varepsilon}{2}}e^{-t}t^{\frac{d}{2}-1} \Big(\ln \frac{1}{t}\Big)^{-2\gamma}h(t)dt 
\\
\ll\varepsilon^{\frac{d}{2}}\int\limits_{0}^{\frac{\varepsilon}{2}}e^{-t}t^{-1} \Big(\ln \frac{1}{t}\Big)^{-2\gamma}h(t)dt <
\varepsilon^{\frac{d}{2}}\int\limits_{0}^{\frac{\varepsilon}{2}}t^{-1} \Big(\ln \frac{1}{t}\Big)^{-2\gamma}h(t)dt  \\
=\frac{1}{2\gamma-1}\varepsilon^{\frac{d}{2}}\Big(\ln \frac{2}{\varepsilon}\Big)^{-2\gamma+1}\ll\varepsilon^{\frac{d}{2}}(\ln N)^{-2\gamma+1}.
\end{multline}

Formulas (\ref{wceLaplaceH}), (\ref{large_t}), (\ref{medium_t}) and (\ref{small_t}) imply
\begin{align*}
\mathrm{wce}(Q[X_{N}];\mathbb{H}^{(\frac{d}{2},\gamma)}(\mathbb{S}^{d}))^{2}\ll
\varepsilon^{s}+
\varepsilon^{\frac{d}{2}}(\ln N)^{-2\gamma}+
 \varepsilon^{\frac{d}{2}}(\ln N)^{-2\gamma+1}\ll \varepsilon^{\frac{d}{2}}(\ln N)^{-2\gamma+1}
 \end{align*}
and  Lemma~\ref{lem4} is proved.
\end{proof}

 \begin{proof}[Proof of Theorem~\ref{thm5}] 
If $X_{N}$ is a sequence of $N$-point QMC designs for  $\mathbb{H}^{s}(\mathbb{S}^{d})$, $s>\frac{d}{2}$, then by  (\ref{QmcHsDefinition}) and (\ref{Lemma4}) 

\begin{multline*}
\mathrm{wce}(Q[X_{N}];\mathbb{H}^{(\frac{d}{2},\gamma)}(\mathbb{S}^{d}))<C(d,s,\gamma)
[\mathrm{wce}(Q[X_{N}];\mathbb{H}^{s}(\mathbb{S}^{d}))]^{\frac{d}{2s}}(\ln N)^{-\gamma+\frac{1}{2}} \\
\ll (N^{-\frac{s}{d}})^{\frac{d}{2s}}(\ln N)^{-\gamma+\frac{1}{2}}=N^{-\frac{1}{2}}(\ln N)^{-\gamma+\frac{1}{2}}
 \end{multline*}
and  Theorem~\ref{thm5} is proved.
 \end{proof}

 \subsection{Examples of QMC designs for classes $\mathbb{H}^{(\frac{d}{2},\gamma)}(\mathbb{S}^{d})$}

 In \cite{Brauchart-Saff-Sloan+2014:qmc_designs} it was shown, that the
 maximisers of the generalised sum of distances
\begin{equation*}
\sum\limits_{i,j=1}^{N}|\mathbf{x}_{i}-\mathbf{x}_{j}|^{2s-d}, \quad N=2,3,4,\ldots
\end{equation*} 
form a sequence of QMC designs for $\mathbb{H}^{s}(\mathbb{S}^{d})$ for $s$ in
the interval $(\frac{d}{2}, \frac{d}{2}+1)$.

Consequently, from this fact and from Theorem~\ref{thm5} we obtain the
statement.

\begin{thm} Let $\gamma>\frac{1}{2}$ and $0<\alpha<2$. Then, the maximisers of
  generalised sum of distances
\begin{equation*}
  \sum\limits_{i,j=1}^{N}|\mathbf{x}_{i}-\mathbf{x}_{j}|^{\alpha}, \quad N=2,3,4,\ldots
\end{equation*} 
form a sequence of QMC designs for
$\mathbb{H}^{(\frac{d}{2},\gamma)}(\mathbb{S}^{d})$.
\end{thm}

\begin{thm} If $X_{N}^{*}$, $N=2,3,\ldots,$ minimises the energy functional
    \begin{equation*}
    \sum\limits_{i,j=1}^{N}\tilde{K}_{\gamma,d}(\mathbf{x}_{i},\mathbf{x}_{j}),
\end{equation*}    
 where $\tilde{K}_{\gamma,d}(\mathbf{x}, \mathbf{y})$   is defined by (\ref{kernel}), then there exists $C_{d,\gamma}>0$, such that for all $N\geq 2$
 \begin{equation*}
\mathrm{wce}(Q[X^{*}_{N}];\mathbb{H}^{(\frac{d}{2},\gamma)}(\mathbb{S}^{d}))\leq \frac{C_{d,\gamma}}{N^{\frac{1}{2}}(\ln N)^{\gamma-\frac{1}{2}}}.
\end{equation*}
 Consequently, $X_{N}^{*}$ is a sequence of QMC designs for 
 $\mathbb{H}^{(\frac{d}{2},\gamma)}(\mathbb{S}^{d})$.
 \end{thm}


%

\section*{References}

\providecommand{\bysame}{\leavevmode\hbox to3em{\hrulefill}\thinspace}
\providecommand{\MR}{\relax\ifhmode\unskip\space\fi MR }
\providecommand{\MRhref}[2]{%
  \href{http://www.ams.org/mathscinet-getitem?mr=#1}{#2}
}
\providecommand{\href}[2]{#2}

\end{document}